\documentclass[12pt,reqno]{amsart}
\usepackage{enumerate}
\usepackage{amsrefs}
\usepackage{amsfonts, amsmath, amssymb, amscd, amsthm, bm, cancel}
\usepackage{url}
\usepackage{graphicx}
\usepackage[%backref=page,
linktocpage=true,colorlinks,citecolor=magenta,linkcolor=blue,urlcolor=magenta]{hyperref}
\usepackage{multicol}
\usepackage{comment}
\usepackage{mathrsfs}
\usepackage[T1]{fontenc}
\usepackage[margin=1in]{geometry}
\usepackage{appendix}
\usepackage{tikz}
\usetikzlibrary{arrows.meta}
\usetikzlibrary{chains}
\usetikzlibrary{patterns}

 \newtheorem{thm}{Theorem}[section]
  
 \newtheorem{cor}[thm]{Corollary}
  
 \newtheorem{lem}[thm]{Lemma}

 \theoremstyle{definition}
 
  \newtheorem*{ack}{Acknowledgments}
 \theoremstyle{remark}
 \newtheorem{rem}[thm]{Remark}

 \numberwithin{equation}{section}
 \allowdisplaybreaks

\begin{document}
\title[The planar isotropic $L_p$ dual Minkowski problem]{Classification of solutions for the planar isotropic\\
$L_p$ dual Minkowski problem}

\author{Haizhong Li}
\address{Department of Mathematical Sciences, Tsinghua University, Beijing 100084, P.R. China}
\email{\href{mailto:lihz@tsinghua.edu.cn}{lihz@tsinghua.edu.cn}}

\author{Yao Wan}
\address{Department of Mathematical Sciences, Tsinghua University, Beijing 100084, P.R. China}
\email{\href{mailto:y-wan19@mails.tsinghua.edu.cn}{y-wan19@mails.tsinghua.edu.cn}}


%%% ----------------------------------------------------------------------

\begin{abstract}
In his beautiful paper \cite{Ben03}, Ben Andrews obtained the complete classification of the solutions of the planar isotropic $L_p$ Minkowski problem. In this paper, by generalizing Ben Andrews's result we obtain the complete classification of the solutions of the planar isotropic $L_p$ dual Minkowski problem, that is, for any $p,q\in\mathbb{R}$ we obtain the complete classification of the solutions of the following equation:
\begin{equation*}
    u^{1-p}(u_{\theta}^2+u^2)^{\frac{q-2}{2}}(u_{\theta\theta}+u)=1\quad\text{on}\ \mathbb{S}^1.
\end{equation*}
To establish the classification, we convert the ODE for the solution into an integral and study its asymptotic behavior, duality and monotonicity.
\end{abstract}

\maketitle
%\tableofcontents

\section{Introduction}

The Brunn-Minkowski theory is the classical core of the convex geometry, which focuses on studying geometric measures and geometric inequalities of convex bodies, see, for example Schneider \cite{Sch14}. The Minkowski problem is of central importance in the Brunn-Minkowski theory, which asks to characterize the surface area measure. Besides, the Aleksandrov problem characterizes the integral measure, and the dual Aleksandrov problem characterizes the dual integral measure. These problems are influential problems not only in geometric analysis, but also in the theory of fully nonlinear partial differential equations.

The $L_p$ Brunn-Minkowski theory and the dual Brunn-Minkowski theory are two theories that fundamentally extended the classical Brunn-Minkowski theory, which were initiated by Lutwak \cite{Lut7501, Lut7502, Lut93, Lut96}. Lutwak \cite{Lut93, Lut96} introduced the $L_p$ surface measure
and proposed the $L_p$ Minkowski problem, which asks to characterize the $L_p$ surface area measure. The $L_p$ Minkowski problem contains critical problems such as the classical Minkowski problem ($p=1$), the logarithmic Minkowski problem ($p=0$) and the centro-affine Minkowski problem ($p=-n$) as special cases, and has been extensively studied, see, for example \cite{BLY13, BT17, CY76, CW06, HLY05, JLZ16, LW13, LYZ04}. Huang, Lutwak, Yang and Zhang \cite{HLY16} introduced the dual curvature measure and proposed the dual Minkowski problem, which asks to characterize the dual curvature measure. The dual Minkowski problem contains critical problems such as the logarithmic Minkowski problem ($q=n$) and the Aleksandrov problem ($q=0$) as special cases, and has been studied in \cite{BHP18, BLY19, CL18, HP18, Zhao17, Zhao18}.

The $L_p$ dual Minkowski problem was posed by Lutwak, Yang and Zhang in \cite{LYZ18}, and unifies all Minkowski type problems mentioned so far. The $L_p$ dual Minkowski problem contains critical problems such as the $L_p$ Minkowski problem ($q=n$), the dual Minkowski problem ($p=0$) and the $L_p$ Aleksandrov problem ($q=0$) as special cases, and has been intensively studied by many authors, see, for example \cite{BF19, CCL21, CHZ19, CL21, HZ18, LLL22, SX21}.

When the given measure $\mu$ has a density $f$, the $L_p$ dual Minkowski problems becomes the following Monge-Ampère type equation on $\mathbb{S}^{n-1}$:
\begin{equation}\label{1.1}
    u^{1-p}(u^2+|\nabla u|^2)^{\frac{q-n}{2}}\det(\nabla^2 u+uI)=f,
\end{equation}
where $f$ is a given positive smooth function on $\mathbb{S}^{n-1}$, $u$ is the support function of a convex body, and $\nabla u$ and $\nabla^2 u$ are the gradient and the Hessian of $u$ on the unit sphere with respect to an orthonormal basis, respectively. 

We call the $L_p$ dual Minkowski problem with $f=1$ in $\mathbb{R}^2$ the planar isotropic $L_p$ dual Minkowski problem, which is concerned with the following equation
\begin{equation}
    u^{1-p}(u_{\theta}^2+u^2)^{\frac{q-2}{2}}(u_{\theta\theta}+u)=1, \label{1.2}
\end{equation}
where $u:\mathbb{S}^1\to(0,+\infty)$ is the support function of a closed planar curve. The classification of solutions for Eq.(\ref{1.2}) is of great interest to the study of the $L_p$ dual Minkowski problem. We call a solution $u$ of (\ref{1.2}) to be an embedded  (immersed, resp.) solution if it corresponds to the support function of an embedded (immersed, resp.) curve.

The aim of this paper is to classify the solutions of Eq.(\ref{1.2}) for any $p,q\in\mathbb{R}$. When $q=2$, the planar $L_p$ dual Minkowski problem becomes the planar $L_p$ Minkowski problem, and Eq.(\ref{1.2}) can be deduced to the following equation
\begin{equation}
	u^{1-p}(u_{\theta\theta}+u)=1.\label{1.3}
\end{equation}
If $p\neq1$, let $\alpha=\frac{1}{1-p}$, then Eq.(\ref{1.3}) is equivalent to
\begin{equation}
    \kappa^{\alpha}=\langle x,\nu\rangle,\label{1.4}
\end{equation}
where $x$ is a closed planar curve with outer unit normal $\nu$ and curvature $\kappa$. This equation is closely related to the asymptotic behavior of the generalized curve shortening problem. The classification of the embedded solutions of Eq.(\ref{1.4}) has been completely obtained by Ben Andrews in his beautiful paper \cite{Ben03}.

\begin{thm}[\cite{Ben03}]\label{Thm1.1}
If $\alpha\in[\frac{1}{8},\infty)$, then the only embedded solution of Eq.(\ref{1.4}) is a circle with radius $1$ centred at the origin except for $\alpha=\frac{1}{3}$ where they could be ellipses. If $\alpha\in(0,\frac{1}{8})$, then the embedded solution is a circle  with radius $1$ centred at the origin, as well as unique (up to rotation and scaling) curves $\Gamma_{k,\alpha}$ with $k$-fold symmetry, for each integer $k$ satisfying $3\le k<\sqrt{\frac{\alpha+1}{\alpha}}$. The curves $\Gamma_{k,\alpha}$ depend smoothly on $\alpha<\frac{1}{k^2-1}$ and converge to regular $k$-sided polygons as $\alpha$ tends to $0$, and to  circles as $\alpha$ tends to $\frac{1}{k^2-1}$.
\end{thm}

For the general case of Eq.(\ref{1.2}), it is known that the solution is unique when $q<p$ \cite{HZ18} and is unique up to scaling when $p=q$ \cite{CL21}, which can be obtained by the strong maximum principle. When $-2\le p<q\le \min\{2,p+2\}$, Chen-Huang-Zhao \cite{CHZ19} showed that the even solution of Eq.(\ref{1.2}) must be constant. When $1<p<q\le 2$, Chen-Li \cite{CL21} showed that the solution of Eq.(\ref{1.2}) must be constant. 

On the other hand, the nonuniqueness of solutions of Eq.(\ref{1.2}) has been provided by Huang-Jiang \cite{HJ19} when $p=0$ and an even $q\ge 6$, and by Chen-Chen-Li \cite{CCL21} either when $pq\ge0,\ q-p>4$, or when $p<0<q,\ 1+\frac{1}{p}<\frac{1}{q}$ and $q-p>4$. It was proved by Jiang-Wang-Wu \cite{JWW21} that when $p<0$ and an even $q\ge2$ are such that $q-p>16$, then there exists a positive classical solution $u$ of Eq.(\ref{1.2}) with the least period $\frac{2\pi}{m}$ for any integer $m\in [4,\sqrt{q-p})$.

Recently, Liu-Lu \cite{LL22} already studied the dual Minkowski problem ($p=0$) and proved that the solution of Eq.(\ref{1.2}) must be constant when $p=0$ and $1\le q \le 4$, there exists exactly two solutions (up to rotation) when $p=0$ and $0<q<1$, and there exists at least $\left\lceil \sqrt{q}\right\rceil-1$ solutions when $p=0$ and $q>4$. By duality, they also proved that the solution of Eq.(\ref{1.2}) must be constant when $q=0$ and $-4\le p \le -1$, there exists exactly two solutions (up to rotation) when $q=0$ and $-1<p<0$, and there exists at least $\left\lceil \sqrt{-p}\right\rceil-1$ solutions when $q=0$ and $p<-4$.

In this paper, we obtain the complete classification of the embedded solutions of Eq.(\ref{1.2})  for any $p,q\in\mathbb{R}$.

\begin{thm} \label{Thm1.2}
The classification of the embedded solutions of Eq.(\ref{1.2}) is as follows:

\begin{enumerate}
    \item[Case (1)] \quad $q-p\le 0$.
    \begin{enumerate}
        \item[Subcase $1^\circ$] If $q<p$, then the embedded solution is unique.
        \item[Subcase $2^\circ$] If $q=p$, then the embedded solution is unique up to scaling.
        
        $\ $
    \end{enumerate}
    
    \item[Case (2)] \quad $0<q-p\le 1$.
    \begin{enumerate}
        \item[Subcase $1^\circ$] If $(p,q)=(1,2)$, then each embedded solution has the form 
        \begin{align*}
        u(\theta)=1+\lambda\cos(\theta-\theta_0),\quad\lambda\ge0,\ \theta_0\in[0,2\pi),
        \end{align*}
        which is unique up to translation transformations. Note that $u(\theta)$ is the support function of a circle with radius 1.
        
        \item[Subcase $2^\circ$] If $(p,q)=(-2,-1)$, then each embedded solution has the form 
        \begin{align*}
        u(\theta)=\frac{\sqrt{1-\mu^2\sin^2(\theta-\theta_0)}-\mu\cos(\theta-\theta_0)}{1-\mu^2},\quad 0\le\mu<1,\ \theta_0\in[0,2\pi).
        \end{align*}
        Note that $\frac{1}{u(\theta)}$ is the radial function of a circle which has radius $1$ and contains the origin in its interior.
        
        \item[Subcase $3^\circ$] If $q-p=1$ and $(p,q)\neq (1,2),\ (-2,-1)$, then the embedded solution is unique. 
        
        \item[Subcase $4^\circ$] If $q-p<1$ and $q\le 2p$, then the embedded solution is unique. 
        
        \item[Subcase $5^\circ$] If $q-p<1,\ q>2p$ and $2q\le p$, then the embedded solution is unique. 
        
        \item[Subcase $6^\circ$] If $q-p<1,\ q>2p$ and $2q>p$, then there exist exactly two embedded solutions.
        
        $\ $
    \end{enumerate}
    
    \item[Case (3)] \quad $1<q-p\le 4$.
    \begin{enumerate}
        \item[Subcase $1^\circ$] If $(p,q)=(-2,2)$, then each embedded solution has the form 
        $$u(\theta)=\sqrt{\lambda^2\cos^2(\theta-\theta_0)+\lambda^{-2}\sin^2(\theta-\theta_0)},\quad \lambda>0,\ \theta_0\in[0,2\pi),$$
        which is unique up to rotation and unimodular affine transformations. Note that $u(\theta)$ is the support function of an ellipse with area $1$ and centered at the origin.
        
        \item[Subcase $2^\circ$] If $q\ge 2p$, $2q\ge p$ and $1<q-p\le 3$, then the embedded solution is unique.
        
        \item[Subcase $3^\circ$] If $q\ge 2p$, $2q\ge p,\ 3<q-p\le 4$ and $(p,q)\neq(-2,2)$.
        
        \item[Subcase $3.1^\circ$] If in addition $p\ge -2$ and $q\le 2$ , then the embedded solution is unique.
        
        \item[Subcase $3.2^\circ$] If in addition $q>2$ and $p\ge\frac{2q}{2+q}$, then the embedded solution is unique.
        
        \item[Subcase $3.3^\circ$] If in addition $q>2$ and $p<\frac{2q}{2+q}$, then the non-constant embedded solution must be $\pi$-periodic if it exists.
        
        \item[Subcase $3.4^\circ$] If in addition $p<-2$ and $q\le \frac{2p}{2-p}$, then the embedded solution is unique.
        
        \item[Subcase $3.5^\circ$] If in addition $p<-2$ and $q> \frac{2p}{2-p}$, then the non-constant embedded solution must be $\pi$-periodic if it exists.
        
        \item[Subcase $4^\circ$] If $q\ge 2p$ and $2q< p$, then there exist exactly two embedded solutions.
        
        \item[Subcase $5^\circ$] If $q<2p$, then there exist exactly two embedded solutions.
        
        $\ $
    \end{enumerate}

    \item[Case ($4$)] \quad $(k-1)^2<q-p\le k^2,\quad k\ge3$.
    \begin{enumerate}
        \item[Subcase $1^\circ$] If $q\ge 2$ and $p\le -2$, then there exist exactly $(k-2)$ embedded solutions.
        
        \item[Subcase $2^\circ$] If $q\ge 2,\ p<\frac{2q}{2+q}$ and $-2<p\le -1$, then there exist at least $(k-2)$ embedded solutions.
        
        \item[Subcase $2.1^\circ$] If in addition $p>\frac{q}{1-q}$ and $k=3$, then there exist at least two embedded solutions.
        
        \item[Subcase $3^\circ$] If $q\ge 2,\ p<\frac{2q}{2+q}$ and $p>-1$, then there exist at least $(k-1)$ embedded solutions.
        
        \item[Subcase $4^\circ$] If $q\ge 2,\ p\ge \frac{2q}{2+q},\ p>-2$ and $q\ge 2p$, then there exist exactly $(k-1)$ embedded solutions.
        
        \item[Subcase $5^\circ$] If $q\ge 2,\ p\ge \frac{2q}{2+q},\ p>-2$ and $q<2p$, then there exist exactly $k$ embedded solutions.
        
        \item[Subcase $6^\circ$] If $p\le -2,\ q>\frac{2p}{2-p}$ and $1\le q<2$, then there exist at least $(k-2)$ embedded solutions.
        
        \item[Subcase $6.1^\circ$] If in addition $q<\frac{p}{1+p}$ and $k=3$, then there exist at least two embedded solutions.
        
        \item[Subcase $7^\circ$] If $p\le -2,\ q>\frac{2p}{2-p}$ and $q<1$, then there exist at least $(k-1)$ embedded solutions.
        
        \item[Subcase $8^\circ$] If $p\le -2,\ q\le \frac{2p}{2-p},\ q< 2$ and $2q\ge p$, then there exist exactly $(k-1)$ embedded solutions.
        
        \item[Subcase $9^\circ$] If $p\le -2,\ q\le \frac{2p}{2-p},\ q<2$ and $2q<p$, then there exist exactly $k$ embedded solutions.
        
    \end{enumerate}
\end{enumerate}

Furthermore, if $p<q$ and $(p,q)\neq(1,2),\ (-2,-1)$ or $(-2,2)$, then any non-constant embedded solution corresponds to a curve $\Gamma_{k,p,q}$ with $k$-fold symmetry for some integer $k$ strictly between $\sqrt{q-p}$ and $\Xi(p,q)$, where
\begin{align*}
    \Xi(p,q)=\left\{\begin{array}{l}
    \frac{2(q-p)}{q},\quad  \text{if}\ 0\le p<q\\
    2,\quad\quad\quad \text{if}\ p<0<q\\ 
    \frac{2(p-q)}{p},\quad \text{if}\ p<q\le 0.
    \end{array}\right.
\end{align*}
\end{thm}

\begin{rem}
	When $q=2$ and $p<1$, our Theorem \ref{Thm1.2} reduces to Theorem \ref{Thm1.1} in \cite{Ben03}.
\end{rem}

\begin{rem}
	When $p=0$ (or $q=0$), a stronger result was obtained by Liu-Lu in \cite{LL22}.
\end{rem}

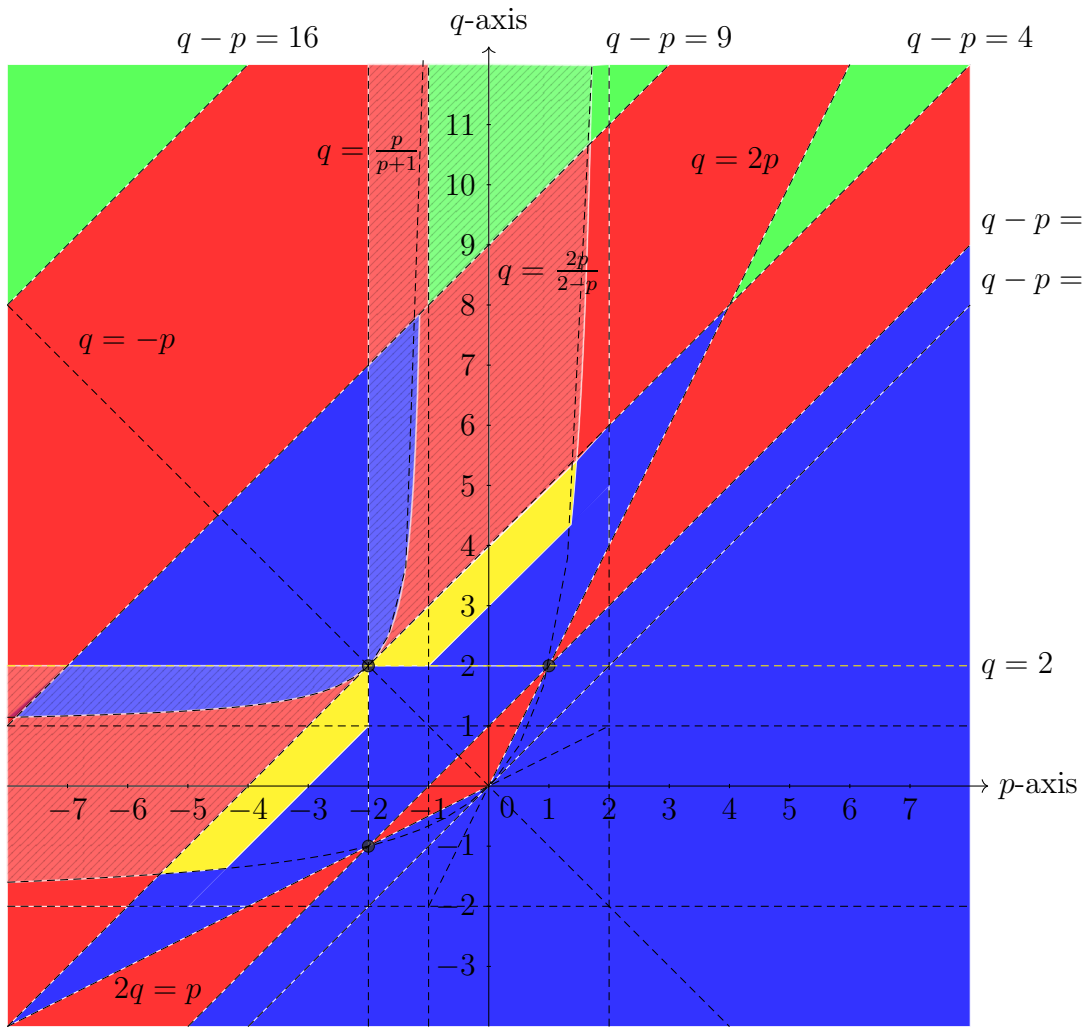
\begin{figure}[htbp]
\centering
\begin{tikzpicture}[scale=0.8][>=Stealth] 
%\draw[->](-8,0)--(8.3,0) node[right]{$p$-axis};
%\draw[->](0,-4)--(0,12.3) node[above]{$q$-axis};
%number=1
\filldraw[opacity=0.8, draw=white, fill=blue!100] (-4,-4)--(8,8)--(8,-4)--cycle;
\filldraw[opacity=0.8, draw=white, fill=blue!100] (-4,-4)--(0,0)--(-2,-1)--(-5,-4)--cycle;
\filldraw[opacity=0.8, draw=white, fill=blue!100] (1,2)--(0,0)--(8,8)--(8,9)--cycle;
\filldraw[opacity=0.8, draw=white, fill=blue!100] (-2,-1)--(1,2)--(-2,2)--cycle;
\filldraw[opacity=0.8, draw=white, fill=blue!100] (-1,2)--(1,2)--(2,4)--(2,5)--cycle;
\filldraw[opacity=0.8, draw=white, fill=blue!100] (2,4)--(4,8)--(2,6)--cycle;
\filldraw[opacity=0.8, draw=blue!80, fill=blue!100, domain=1.37:1.46] plot(\x,{2*\x/(2-\x)})--(2,6)--(2,5)--cycle;
\filldraw[opacity=0.8, draw=white, fill=blue!100] (-2,1)--(-2,-1)--(-4,-2)--(-5,-2)--cycle;
\filldraw[opacity=0.8, draw=white, fill=blue!100] (-4,-2)--(-8,-4)--(-6,-2)--cycle;
\filldraw[opacity=0.8, draw=blue!80, fill=blue!100, domain=-5.46:-4.37] plot(\x,{2*\x/(2-\x)})--(-5,-2)--(-6,-2)--cycle;
\filldraw[opacity=0.8, draw=white, fill=blue!100] (-2,2)--(-7,2)--(-2,7)--cycle;
\filldraw[opacity=0.6, draw=white, fill=blue!100, domain=-2:-1.145] plot(\x,{\x/(1+\x)})--(-2,7)--(-2,2)--cycle;
\filldraw[thick, pattern color=black, pattern=north east lines, opacity=0.4, draw=blue!20, domain=-2:-1.145] plot(\x,{\x/(1+\x)})--(-2,7)--(-2,2)--cycle;
\filldraw[opacity=0.6, draw=white, fill=blue!100, domain=-8:-2] plot(\x,{\x/(1+\x)})--(-2,2)--(-7,2)--cycle;
\filldraw[thick, pattern color=black, pattern=north east lines, opacity=0.4, draw=blue!20, domain=-8:-2] plot(\x,{\x/(1+\x)})--(-2,2)--(-7,2)--cycle;
%number=2
\filldraw[opacity=0.8, draw=white, fill=red!100] (-2,-1)--(0,0)--(1,2)--cycle;
\filldraw[opacity=0.8, draw=white, fill=red!100] (1,2)--(8,9)--(8,12)--(4,8)--cycle;
\filldraw[opacity=0.8, draw=white, fill=red!100] (2,6)--(4,8)--(6,12)--(3,12)--(2,11)--cycle;
\filldraw[opacity=0.8, draw=white, fill=red!100, domain=1.46:1.685] plot(\x,{2*\x/(2-\x)})--(2,11)--(2,6)--cycle;
\filldraw[opacity=0.8, draw=white, fill=red!100] (-2,-1)--(-5,-4)--(-8,-4)--cycle;
\filldraw[opacity=0.8, draw=white, fill=red!100] (-6,-2)--(-8,-2)--(-8,-4)--cycle;
\filldraw[opacity=0.8, draw=white, fill=red!100, domain=-8:-5.46] plot(\x,{2*\x/(2-\x)})--(-6,-2)--(-8,-2)--(-8,-1.6)--cycle;
\filldraw[opacity=0.8, draw=white, fill=red!100] (-8,2)--(-7,2)--(-2,7)--(-2,12)--(-4,12)--(-8,8)--cycle;
\filldraw[opacity=0.6, draw=white, fill=red!100, domain=1.46:1.685] plot(\x,{2*\x/(2-\x)})--(0,9)--(0,4)--cycle;
\filldraw[thick, pattern color=black, pattern=north east lines, opacity=0.4, draw=red!20, domain=1.46:1.685] plot(\x,{2*\x/(2-\x)})--(0,9)--(0,4)--cycle;
\filldraw[opacity=0.6, draw=white, fill=red!100, domain=-2:-1.145] plot(\x,{\x/(1+\x)})--(0,9)--(0,4)--cycle;
\filldraw[thick, pattern color=black, pattern=north east lines, opacity=0.4, draw=red!20, domain=-2:-1.145] plot(\x,{\x/(1+\x)})--(0,9)--(0,4)--cycle;
\filldraw[opacity=0.6, draw=white, fill=red!100, domain=-8:-5.46] plot(\x,{2*\x/(2-\x)})--(-4,0)--(-8,0)--cycle;
\filldraw[thick, pattern color=black, pattern=north east lines, opacity=0.4, draw=red!20, domain=-8:-5.46] plot(\x,{2*\x/(2-\x)})--(-4,0)--(-8,0)--cycle;
\filldraw[opacity=0.6, draw=white, fill=red!100, domain=-8:-2] plot(\x,{\x/(1+\x)})--(-4,0)--(-8,0)--cycle;
\filldraw[thick, pattern color=black, pattern=north east lines, opacity=0.4, draw=red!20, domain=-8:-2] plot(\x,{\x/(1+\x)})--(-4,0)--(-8,0)--cycle;
\filldraw[opacity=0.6, draw=white, fill=red!100] (-1,8)--(-1,12)--(-2,12)--(-2,7)--cycle;
\filldraw[thick, pattern color=black, pattern=north east lines, opacity=0.4, draw=red!20] (-1,8)--(-1,12)--(-2,12)--(-2,7)--cycle;
\filldraw[opacity=0.6, draw=white, fill=red!100] (-8,1)--(-7,2)--(-8,2)--cycle;
\filldraw[thick, pattern color=black, pattern=north east lines, opacity=0.4, draw=red!20] (-8,1)--(-7,2)--(-8,2)--cycle;
%number=3
\filldraw[opacity=0.8, draw=white, fill=green!80] (-8,8)--(-4,12)--(-8,12)--cycle;
\filldraw[opacity=0.8, draw=white, fill=green!80] (4,8)--(8,12)--(6,12)--cycle;
\filldraw[opacity=0.8, draw=white, fill=green!80] (2,11)--(3,12)--(2,12)--cycle;
\filldraw[opacity=0.8, draw=white, fill=green!80, domain=1.685:1.714] plot(\x,{2*\x/(2-\x)})--(2,12)--(2,11)--cycle;
\filldraw[opacity=0.6, draw=white, fill=green!80, domain=1.685:1.714] plot(\x,{2*\x/(2-\x)})--(0,12)--(0,9)--cycle;
\filldraw[thick, pattern color=black, pattern=north east lines, opacity=0.4, draw=green!20, domain=1.685:1.714] plot(\x,{2*\x/(2-\x)})--(0,12)--(0,9)--cycle;
\filldraw[opacity=0.6, draw=white, fill=green!80] (-1,8)--(0,9)--(0,12)--(-1,12)--cycle;
\filldraw[thick, pattern color=black, pattern=north east lines, opacity=0.4, draw=green!20] (-1,8)--(0,9)--(0,12)--(-1,12)--cycle;
%number=1
\filldraw[opacity=0.8, draw=white, fill=yellow!100, domain=1.37:1.46] plot(\x,{2*\x/(2-\x)})--(-2,2)--(-1,2)--cycle;
\filldraw[opacity=0.8, draw=white, fill=yellow!100, domain=-5.46:-4.37] plot(\x,{2*\x/(2-\x)})--(-2,1)--(-2,2)--cycle;
%number=infinity
\filldraw[opacity=0.8, fill=black!80] (-2,2) circle (0.1);
\filldraw[opacity=0.8, fill=black!80] (1,2) circle (0.1);
\filldraw[opacity=0.8, fill=black!80] (-2,-1) circle (0.1);
%\filldraw[opacity=0.6, draw=white, fill=green!20, domain=-1.145:-1.09] plot(\x,{\x/(1+\x)})--(-1,12)--(-1,8)--cycle;
%\filldraw[pattern color=black, pattern=north east lines, opacity=0.4, draw=green!20, domain=-1.145:-1.09] plot(\x,{\x/(1+\x)})--(-1,-12)--(-1,8)--cycle;
\draw[->](-8,0)--(8.3,0) node[right]{$p$-axis};
\draw[->](0,-4)--(0,12.3) node[above]{$q$-axis};
\foreach \x in {-7, -6, -5, -4, -3, -2, -1, 1, 2, 3, 4, 5, 6, 7} \draw (\x, 1pt) -- (\x, -1pt) node[anchor=north] {$\x$};
\foreach \y in {-3, -2, -1, 1, 2, 3, 4, 5, 6, 7, 8, 9, 10, 11} \draw (1pt, \y) -- (-1pt, \y) node[anchor=east] {$\y$};
\node[below right] at (0,0) {$0$};
\draw [densely dashed, draw=yellow] (-8,2)--(8,2);
\node[ right] at (8,2) {$q=2$};
\draw [densely dashed] (-8,1)--(8,1);
\draw [densely dashed] (-2,-4)--(-2,12);
\draw [densely dashed] (-1,-4)--(-1,12);
\draw [densely dashed] (2,-4)--(2,12);
\draw [densely dashed] (-8,-2)--(8,-2);
\draw [densely dashed] (-8,8)--(4,-4);
\node[above right] at (-7,7) {$q=-p$};
\draw [densely dashed] (-4,-4)--(8,8);
\node[above right] at (8,8) {$q-p=0$};
\draw [densely dashed] (-5,-4)--(8,9);
\node[above right] at (8,9) {$q-p=1$};
\draw [densely dashed] (-8,-4)--(8,12);
\node[above ] at (8,12) {$q-p=4$};
\draw [densely dashed] (-8,1)--(3,12);
\node[above ] at (3,12) {$q-p=9$};
\draw [densely dashed] (-8,8)--(-4,12);
\node[above ] at (-4,12) {$q-p=16$};
\draw [densely dashed] (-1,-2)--(6,12);
\node[above left] at (5,10) {$q=2p$};
\draw [densely dashed] (-8,-4)--(2,1);
\node[below ] at (-5.5,-3) {$2q=p$};
\draw [densely dashed] [domain = -8: 1.714] plot ({\x},{2*\x/(2-\x)});
\node[above ] at (1,8) {$q=\frac{2p}{2-p}$};
\draw [densely dashed] [domain = -8: -1.09] plot ({\x},{\x/(1+\x)});
\node[above ] at (-2,10) {$q=\frac{p}{p+1}$};
%\draw[rotate=-135] (0,-1.414) ellipse (2.449 and 1.414);
%\draw[domain = 0: 6.29] plot ({4*cos(\x r)*cos(\x r)-3},{2*cos(\x r)*cos(\x r)+3*sin(\x r)*cos(\x r)});
\end{tikzpicture}
\caption{Number $k$ of embedded solutions of Eq.(\ref{1.2}) \\
Blue domain: $k=1$. Red domain: $k=2$. Green domain: $k=3$. Blue domain with shadow: $k\ge1$. Red domain with shadow: $k\ge2$. Green domain with shadow: $k\ge3$. Yellow domain: a non-constant solution must be $\pi$-periodic.}
\label{fig-1}
\end{figure}

To establish the classification, we convert Eq.(\ref{1.2}) for the embedded solution into an integral $\Theta(p,q,r)$ and study its asymptotic behavior, duality and monotonicity. This argument allows a classification of solutions without the assumption of embeddedness.

\begin{thm}\label{Thm1.3}
We consider the immersed solutions of Eq.(\ref{1.2}).
	
\begin{enumerate}
	\item[Case (1)] If $p\ge q$, all immersed solutions are the constant solutions.
	
	\item[Case (2)] If $(p,q)=(1,2),\ (-2,-1)$ or $(-2,2)$, all immersed solutions are given in Case (2) $1^\circ$, Case (2) $2^\circ$ and  Case (3) $1^\circ$ of Theorem \ref{Thm1.2}, respectively.
	
	\item[Case (3)] Let $m$ and $n$ be mutually prime. If $p<q$ and $(p,q)\neq(1,2),\ (-2,-1)$ or $(-2,2)$, there exists a smooth, strictly locally convex, non-constant immersed solution, which corresponds to a curve $\Gamma_{n,m,p,q}$ with total curvature $2\pi n$ and $m$ maxima of solution if $\frac{m}{n}$ lies strictly between $\sqrt{q-p}$ and $\Xi(p,q)$. Moreover, if in addition $(p,q)$ satisfies one of the following cases:
	\begin{enumerate}
		\item[(i)]\quad $ p\le -2,\ q\ge 2$, or
		\item[(ii)]\quad $p\le -2,\ q\le\frac{2p}{2-p}$, or
		\item[(iii)]\quad $q\ge 2,\ p\ge\frac{2q}{2+q}$, or
		\item[(iv)]\quad $p\ge -2,\ q\le 2,\ q\ge\frac{2p}{2-p}$,
	\end{enumerate}
	then there exists exactly one such solution up to rotation and scaling.

\end{enumerate}
\end{thm}

$\ $

The paper is organized as follows. In section \ref{sec:2}, we construct a period function $\Theta(p,q,r)$ which is closely related to the solution of Eq.(\ref{1.2}). In section \ref{sec:3}, we compute some special values of $\Theta(p,q,r)$ and find out asymptotic behaviours as $r$ tends to $1$ and infinity. In section \ref{sec:4}, we discover some duality relations of $\Theta(p,q,r)$, which is of great use to simplify discussion. In section \ref{sec:5}, we prove that $\Theta(p,q,r)$ is monotone in $p$ and $q$. In section \ref{sec:6}, we prove that $\Theta(p,q,r)$ is monotone in $r$ for some $(p,q)$. In section \ref{sec:7}, we give the proofs of Theorem \ref{Thm1.2} and Theorem \ref{Thm1.3}.

\begin{ack}
This work was supported by NSFC grant No.11831005, NSFC Grant No.12126405 and NSFC-FWO Grant No. 11961131001.
\end{ack}

\section{The period function}
\label{sec:2}

Let $u(\theta)$ be the support function of the solution for the planar isotropic $L_p$ dual Minkowski problem, that is, $u$ satisfies the following second-order ordinary differential equation

\begin{equation}
    u^{1-p}(u_{\theta}^2+u^2)^{\frac{q-2}{2}}(u_{\theta\theta}+u)=1.\label{2.1}
\end{equation}

For $q\neq0$, notice that (\ref{2.1}) can be rewritten as 
\begin{align*}
	&  \frac{2}{q}((u_{\theta}^2+u^2)^{\frac{q}{2}})_{\theta}=\frac{2}{p}(u^p)_{\theta}, & p\neq0,\\
	&  \frac{2}{q}((u_{\theta}^2+u^2)^{\frac{q}{2}})_{\theta}=2(\log u)_{\theta}, & p=0,
\end{align*}
thus a first integral of (\ref{2.1}) is given by
\begin{equation}\label{2.2}
\begin{split}
    &  (u_{\theta}^2+u^2)^{\frac{q}{2}}-\frac{q}{p}u^p=E,\quad\quad\quad\quad  p\neq0,\\
    &  (u_{\theta}^2+u^2)^{\frac{q}{2}}-q\log u=E,\quad\quad\quad  p=0.
\end{split}
\end{equation}
for some constant $E$.

Since the solution of (\ref{2.1}) is a periodic positive function of $\theta$, with maximum and minimum values of $u$ determined by $p,\ q$ and $E$. For convenience, we parameterize the solutions by the ratio $r$ of the maximum and minimum values $u_{\pm}$ of $u$: it follows from (\ref{2.2}) that
\begin{align*}
	&  u_{\pm}^q-\frac{q}{p}u_{\pm}^p=E, & p\neq0,\\
	&  u_{\pm}^q-q\log u_{\pm}=E, & p=0,
\end{align*}
then we have
\begin{equation}\label{2.3}
\begin{split}
    & E=\left(\frac{q(r^p-1)}{p(r^q-1)}\right)^{\frac{p}{q-p}}\left(\frac{q(r^p-r^q)}{p(r^q-1)}\right),\quad p\neq0,\\
    & E=\frac{q\log r}{r^q-1}-\log\left(\frac{q\log r}{r^q-1}\right),\quad p=0.
\end{split}
\end{equation}

It is well known that the solution of (\ref{2.1}) is unique when $p>q$, and is unique up to a constant when $p=q$, see, e.g. \cite{HZ18, CL21}. Hence we only consider the case $p<q$ from now on. 

It follows from (\ref{2.3}) that $E(r)$ is increasing for $r\ge1$ when $q>0$, while $E(r)$ is decreasing for $r\ge 1$ when $q<0$. Moreover, $E\to \frac{p-q}{p}$ as $r\to 1$ for $p\neq0$ and $E\to 1$ as $r\to 1$ for $p=0$; $E\to \infty$ as $r\to \infty$ for $p\le 0$ and $E\to 0$ as $r\to \infty$ for $p>0$.

Using (\ref{2.2}) and (\ref{2.3}), we obtain
\begin{equation}\label{2.4}
\begin{split}
    &  u_{\theta}=\pm \left[\left(\left(\frac{q(r^p-1)}{p(r^q-1)}\right)^{\frac{p}{q-p}}\left(\frac{q(r^p-r^q)}{p(r^q-1)}\right)+\frac{q}{p}u^p\right)^{\frac{2}{q}}-u^2\right]^{\frac{1}{2}},\quad p\neq0,\\
    &  u_{\theta}=\pm \left[\left(\frac{q\log r}{r^q-1}-\log\left(\frac{q\log r}{r^q-1}\right)+q\log u\right)^{\frac{2}{q}}-u^2\right]^{\frac{1}{2}},\quad p=0.
\end{split}
\end{equation}
Thus the total curvature of the unique (up to rotation) smooth, strictly locally convex curve segment with monotone support function satisfying (\ref{2.2}) and endpoints tangent to the circles of radius $1$ and $r$ about the origin is given by
\begin{equation*}
     \Theta(p,q,r) 
     =\int_{\theta_{-}}^{\theta_{+}}d\theta
    =\int_{u_{-}}^{u_{+}}\frac{du}{u_{\theta}}.
\end{equation*}

\begin{enumerate}
    \item If $p\neq0$ and $q\neq0$, denote $x=\frac{u}{u_{-}}\in (1,r)$, the period function $\Theta$ becomes
    \begin{equation}\label{2.5}
    \begin{split}
    \Theta(p,q,r) &=\int_1^r \frac{dx}{\sqrt{\left(\left(\frac{q(r^p-1)}{p(r^q-1)}\right)^{\frac{p}{q-p}}\left(\frac{q(r^p-r^q)}{p(r^q-1)}\right) (u_{-})^{-q}+\frac{q}{p}(u_{-})^{p-q}x^p\right)^{\frac{2}{q}}-x^2}}\\
    &=\int_1^r \frac{ds}{\sqrt{\left(\frac{r^p-r^q}{r^p-1}+\frac{r^q-1}{r^p-1} x^p\right)^{\frac{2}{q}}-x^2}}.
    %=\int_1^r \frac{dx}{\sqrt{\left(1+\frac{r^q-1}{r^p-1} (x^p-1)\right)^{\frac{2}{q}}-x^2}}.
    \end{split}
    \end{equation}
    
    \item  If $p=0$, denote $x=\frac{u}{u_{-}}\in (1,r)$, the period function $\Theta$ becomes
    \begin{equation}\label{2.6}
    \begin{split}
    \Theta(0,q,r) 
    &=\int_1^r \frac{dx}{\sqrt{\left(
    \left(\frac{q\log r}{r^q-1}-\log\left(\frac{q\log r}{r^q-1}\right)
    \right) (u_{-})^{-q}
    + (u_{-})^{-q}\log (u_{-})^q+q (u_{-})^{-q}\log x
    \right)^{\frac{2}{q}}-x^2}}\\
    &=\int_1^r \frac{dx}{\sqrt{\left(1+\frac{r^q-1}{\log r}\log x\right)^{\frac{2}{q}}-x^2}}.
    \end{split}
    \end{equation}
    
    \item  If $q=0$, a similar calculation yields
    \begin{align*}
         \log(u_{\theta}^2+u^2)-\frac{2}{p}u^p=E,
    \end{align*}
    where
    \begin{align*}
        E=\frac{2}{p}\log\left(\frac{p\log r}{r^p-1}\right)-\frac{2\log r}{r^p-1}.
    \end{align*}
    Then the period function $\Theta$ becomes
    \begin{equation}\label{2.7}
    \begin{split}
    \Theta(p,0,r) &=\int_1^r \frac{dx}{\sqrt{(u_{-})^{-2}\exp{\left(2\log (u_{-})-\frac{2}{p}(u_{-})^p+\frac{2}{p}(u_{-})^p x^p\right)}-x^2}}\\
    &=\int_1^r \frac{dx}{\sqrt{r^{\frac{2(x^p-1)}{r^p-1}}-x^2}}.
    \end{split}
    \end{equation}
\end{enumerate}

In particular, when $p<q=2$, it follows from (\ref{2.5}) and (\ref{2.6}) that
\begin{equation}
\begin{split}
    & \Theta(p,2,r)=\int_1^r \frac{dx}{\sqrt{\frac{r^p-r^2}{r^p-1}+\frac{r^2-1}{r^p-1} x^p-x^2}},\quad p\neq 0,\\
    & \Theta(0,2,r)=\int_1^r \frac{dx}{\sqrt{1+\frac{r^2-1}{\log r}\log x-x^2}},\quad p=0,
\end{split}
\end{equation}
which are the same as the functions $\Theta\left(\frac{1}{1-p},r\right)$ in \cite{Ben03}.

\section{Special values and asymptotic behaviors}
\label{sec:3}

We only consider the case $(p,q)\in D$, where $D=\{(p,q):\ p<q\}$.

\begin{thm}\label{Thm3.1}
$\Theta$ is continuous on $D\times(1,\infty)$, and has following limiting values:
\begin{align}
    & \label{3.1}
    \lim\limits_{p\to -\infty}\Theta(p,q,r)=\arccos\frac{1}{r},\\
    & \label{3.2}
    \lim\limits_{q\to +\infty}\Theta(p,q,r)=\arccos\frac{1}{r},\\
    & \label{3.3}
    \lim\limits_{r\to 1}\Theta(p,q,r)=\frac{\pi}{\sqrt{q-p}},\\
    & \label{3.4}
    \lim\limits_{r\to \infty}\Theta(p,q,r)=\frac{\pi}{2},& p<0<q,\\
    & \label{3.5}
    \lim\limits_{r\to \infty}\Theta(p,q,r)=\frac{q}{q-p}\cdot\frac{\pi}{2},& 0\le p<q.
   % & \lim\limits_{r\to \infty}\Theta(p,q,r)=\frac{p}{p-q}\cdot\frac{\pi}{2},& p<q\le 0.
\end{align}

Moreover, we also have
\begin{align}
    & \label{3.6}
    \Theta(1,2,r)=\pi, & r>1,\\
    & \label{3.7}
    \Theta(-2,-1,r)=\pi,& r>1,\\
    & \label{3.8}
    \Theta(-2,2,r)=\frac{\pi}{2},& r>1.
\end{align}
\end{thm}

\begin{proof}
At first, it follows from (\ref{2.5}) and (\ref{2.6}) that
\begin{align*}
    & \lim\limits_{p\to -\infty}\Theta(p,q,r)
    =\int_1^r \frac{dx}{\sqrt{r^2-x^2}}=\arccos\frac{1}{r},\\
    & \lim\limits_{q\to +\infty}\Theta(p,q,r)
    =\int_1^r \frac{dx}{\sqrt{r^2-x^2}}=\arccos\frac{1}{r}.
\end{align*}

Now let $\beta\in\mathbb{R}\backslash\{0\}$, and set $x^\beta=v$, where
\begin{align*}
    v=\frac{r^\beta-1}{2}z+\frac{r^\beta+1}{2},\quad\quad z\in (-1,1).
\end{align*}
Then the change of variables formula implies that
\begin{align*}
    \Theta(p,q,r)=\int_{-1}^1\frac{dz}{\sqrt{\tilde{I}(p,q,\beta,r,z)}},
\end{align*}
where
\begin{equation}\label{3.9}
\begin{split}
    & \tilde{I}(p,q,\beta,r,z)=\frac{4\beta^2}{(r^\beta-1)^2}\left[\left(\frac{r^p-r^q}{r^p-1}v^{\frac{q(\beta-1)}{\beta}}+\frac{r^q-1}{r^p-1}v^{\frac{p+q(\beta-1)}{\beta}}\right)^{\frac{2}{q}}-v^2\right],\quad p\neq0,\\
    & \tilde{I}(0,q,\beta,r,z)=\frac{4\beta^2}{(r^\beta-1)^2}\left[v^{\frac{2(\beta-1)}{\beta}}\left(1+\frac{r^q-1}{\beta\log r}\log v\right)^{\frac{2}{q}}-v^2\right].
\end{split}
\end{equation}

When $r$ is close to $1$, a routine computation by Taylor expansions gives
\begin{align}\label{3.10}
    \tilde{I}(p,q,\beta,r,z)
    =(q-p)(1-z^2)\left[1+\frac{z}{2}\left(\beta-\frac{2q-p}{3}\right)(r-1)+o(r-1)\right],
\end{align}
from which it follows that
\begin{align}\label{3.11}
    \lim\limits_{r\to 1}\Theta(p,q,r)=\int_{-1}^1\frac{dz}{\sqrt{(q-p)(1-z^2)}}=\frac{\pi}{\sqrt{q-p}}.
\end{align}
For the completeness of the paper, we provide the proof of (\ref{3.10}) in the Appendix \ref{appendix}.

When $r$ tends to infinity, let $y=\frac{z+1}{2}\in (0,1)$ and choose $\beta=1$, we have
\begin{equation}\label{3.12}
\begin{split}
    \tilde{I}(0,q,1,r,z)= & 4\left[1-y^2+\frac{2\log y}{q\log r}+o\left(\frac{1}{\log r}\right)\right],\\
    \tilde{I}(p,q,1,r,z)= & 4\left[1-y^2+(2-2y)\frac{1}{r}+\frac{2}{q}(1-y^p)r^p+o\left(r^p\right)+o\left(\frac{1}{r}\right)\right] ,\quad p<0<q,\\
    \tilde{I}(p,q,1,r,z)= & 4\left[y^{\frac{2p}{q}}-y^2+\left(\left(2+\frac{2p}{q}\frac{1-y}{y}\right)y^{\frac{2p}{q}}-2y\right)\frac{1}{r}\right.\\
    &\quad \left.+\frac{2}{q}\left( y^{\frac{2p}{q}}-y^{\frac{(2-q)p}{q}}\right)\frac{1}{r^p}+o\left(\frac{1}{r}\right)+o\left(\frac{1}{r^p}\right)\right],\quad 0<p<q,
\end{split}
\end{equation}
from which it follows that
\begin{equation}\label{3.13}
\begin{split}
    & \lim\limits_{r\to \infty}\Theta(p,q,r)=\int_{0}^1\frac{dy}{\sqrt{1-y^2}}=\frac{\pi}{2},\quad  p<0<q,\\
    & \lim\limits_{r\to \infty}\Theta(p,q,r)=\int_{0}^1\frac{dy}{\sqrt{y^{\frac{2p}{q}}-y^2}}=\frac{q}{q-p}\cdot\frac{\pi}{2},\quad 0\le p<q.
\end{split}
\end{equation}

Finally, (\ref{3.6}), (\ref{3.7}) and (\ref{3.8}) can be deduced by direct calculation.
\end{proof}

Next we give the asymptotic behaviors of $\Theta(p,q,r)$ when $r$ tends to infinity.

\begin{thm}\label{Thm3.2}
When $r$ tends to $1$, we have
\begin{align}\label{3.14}
    & \Theta(p,q,r)=\frac{\pi}{\sqrt{q-p}}+\frac{\pi}{96\sqrt{q-p}}(q^2-pq+p^2+3p-3q)(r-1)^2+o\left((r-1)^2\right),\quad q>0.
\end{align}

When $r$ tends to infinity, we have
\begin{align}
    & \label{3.15}
    \Theta(0,q,r)=\frac{\pi}{2}+\frac{\pi}{2q\log r}+o\left(\frac{1}{\log r}\right), & q>0,\\
    & \label{3.16}
    \Theta(p,q,r)=\frac{\pi}{2}+\frac{1}{q}\int_0^1\frac{y^p-1}{(1-y^2)^{\frac{3}{2}}}dy\cdot r^p+o\left(r^p\right), & -1<p<0,\ q>-p. 
\end{align}
\end{thm}

\begin{proof}
If $q>0$, we choose $\beta_0=\frac{2q-p}{3}>0$ because of the asymptotic expansion (\ref{3.10}). Some tedious computation gives rise to
\begin{equation*}
\begin{split}
     \tilde{I}(p,q,\beta_0,r,z)
    &=(q-p)(1-z^2)\left[1+\left(\frac{(2-q)(q-p)}{16}(1-z^2)\right.\right.\\
    &\quad\quad\quad \left.\left.+\frac{(q-2p)(q+p)}{432}(5-z^2)\right)(r-1)^2+o((r-1)^2)\right],
\end{split}
\end{equation*}
as $r$ tends to $1$. Then we deduce
\begin{align*}
    \Theta(p,q,r) &=\int_{-1}^1\frac{dz}{\sqrt{\tilde{I}(p,q,\beta_0,r,z)}}\\
    &=\int_{-1}^1\frac{\left[1+\left(\frac{(2-q)(q-p)}{16}(1-z^2)+\frac{(q-2p)(q+p)}{432}(5-z^2)\right)(r-1)^2+o((r-1)^2)\right]^{-\frac{1}{2}}}{\sqrt{(q-p)(1-z^2)}}dz\\
    &=\int_{-1}^1\frac{1-\frac{1}{2}\left(\frac{(2-q)(q-p)}{16}(1-z^2)+\frac{(q-2p)(q+p)}{432}(5-z^2)\right)(r-1)^2+o((r-1)^2)}{\sqrt{(q-p)(1-z^2)}}dz\\
    &=\frac{\pi}{\sqrt{q-p}}+\frac{\pi}{96\sqrt{q-p}}(q^2-pq+p^2+3p-3q)(r-1)^2+o\left((r-1)^2\right),
\end{align*}
as $r$ tends to $1$.

If $q>0$, it follows from (\ref{3.12}) that
\begin{align*}
    \Theta(0,q,r) &=\int_{-1}^1\frac{dz}{\sqrt{\tilde{I}(0,q,1,r,z)}}
    =\int_0^1\frac{dy}{\sqrt{1-y^2+\frac{2\log y}{q\log r}+o\left(\frac{1}{\log r}\right)}}\\
    &=\int_0^1\frac{1-\frac{\log y}{q(1-y^2)\log r}+o\left(\frac{1}{\log r}\right)}{\sqrt{1-y^2}}dy
    =\frac{\pi}{2}+\frac{\pi}{2q\log r}+o\left(\frac{1}{\log r}\right),
\end{align*}
as $r$ tends to infinity.

If $-1<p<0$ and $q>-p$, we also have
\begin{align*}
    \Theta(p,q,r) &=\int_{-1}^1\frac{dz}{\sqrt{\tilde{I}(p,q,1,r,z)}}
    =\int_0^1\frac{dy}{\sqrt{1-y^2+\frac{2}{q}(1-y^p)r^p+o\left(r^p\right)}}\\
    &=\int_0^1\frac{1-\frac{1-y^p}{q(1-y^2)}\cdot r^p+o\left(r^p\right)}{\sqrt{1-y^2}}dy
    =\frac{\pi}{2}+\frac{1}{q}\int_0^1\frac{y^p-1}{(1-y^2)^{\frac{3}{2}}}dy\cdot r^p+o(r^p),
\end{align*}
as $r$ tends to infinity.

\end{proof}

\section{Duality}
\label{sec:4}

\begin{thm}\label{Thm4.1}
The following duality relations hold:
\begin{align}
    & \label{4.1}
    \Theta\left(\frac{pq}{p-q},q,r^{\frac{q-p}{q}}\right)=\frac{q-p}{q}\Theta(p,q,r), & p<q,\ q>0,\ r>1,\\
    & \label{4.2}
    \Theta\left(p,\frac{pq}{q-p},r^{\frac{p-q}{p}}\right)=\frac{p-q}{p}\Theta(p,q,r), & p<q,\ p<0,\ r>1,\\
    & \label{4.3}
    \Theta(-q,-p,r)=\Theta(p,q,r), & p<q,\ r>1.
\end{align}
\end{thm}

\begin{proof}
Inspired by the observation of Ben Andrews \cite{Ben03} for the case where $q=2$, we prove the duality (\ref{4.1}) as follows. Since the case where $p=0$ in (\ref{4.1}) is trivial, we may assume $p\neq0$.

Denote $\Theta=\Theta(p,q,r)$, and let $u:[0,\Theta]\to\mathbb{R}$ be the monotone solution of (\ref{2.1}) with $u_{\theta}(0)=u_{\theta}(\Theta)=0$ and $u(\Theta)=r u(0)$. Let $w:[0,\beta\Theta]\to\mathbb{R}$ be defined by $w(\tau)=u(\theta)^\beta$, where $\beta=\frac{q-p}{q},\ \tau=\beta\theta$.

Since $u$ satisfies the following equations
\begin{align*}
     u^{1-p}(u_{\theta}^2+u^2)^{\frac{q-2}{2}}(u_{\theta\theta}+u)=1,\quad
     u_{\theta}^2+u^2=\left(\frac{q}{p}u^p+E\right)^\frac{2}{q},
\end{align*}
where $E$ is given by (\ref{2.3}), then we have
\begin{align*}
    w_{\tau\tau}+w 
    &=\frac{\beta-1}{\beta}u^{\beta-2}u_{\theta}^2+\frac{1}{\beta}u^{\beta-1}u_{\theta\theta}+w\\
    &=\frac{\beta-1}{\beta}w^{\frac{\beta-2}{\beta}}\left[\left(\frac{q}{p}w^{\frac{p}{\beta}}+E\right)^{\frac{2}{q}}-w^{\frac{2}{\beta}}\right]
    +\frac{1}{\beta} w^{\frac{\beta-1}{\beta}}\left[w^{\frac{p-1}{\beta}}\left(\frac{q}{p}w^{\frac{p}{\beta}}+E\right)^{\frac{2}{q}-1}-w^{\frac{1}{\beta}}\right]+w\\
    &=\left[\frac{\beta-1}{\beta}E w^{\frac{\beta-2}{\beta}}+\left(\frac{\beta-1}{\beta}\frac{q}{p}+\frac{1}{\beta}\right)w^{\frac{\beta+p-2}{\beta}}\right]\left(\frac{q}{p}w^{\frac{p}{\beta}}+E\right)^{\frac{2}{q}-1}\\
    &=\frac{\beta-1}{\beta}E w^{\frac{\beta-2}{\beta}}\left(\frac{q}{p}w^{\frac{p}{\beta}}+E\right)^{\frac{2-q}{q}},
\end{align*}
and
\begin{align*}
    w_{\tau}^2+w^2
    &=u^{2(\beta-1)}\left(u_{\theta}^2+u^2\right)
    =w^{\frac{2(\beta-1)}{\beta}}\left(\frac{q}{p}w^{\frac{p}{\beta}}+E\right)^\frac{2}{q}.
\end{align*}
    %=\left(E w^{\frac{pq}{p-q}}+\frac{q}{p}\right)^{\frac{2}{q}}.

It follows that
\begin{align}\label{4.4}
    w^{1-\frac{pq}{p-q}}(w_{\tau}^2+w^2)^{\frac{q-2}{2}}(w_{\tau\tau}+w)=\frac{p}{p-q}E.
\end{align}
Note that the left hand side of (\ref{4.4}) is homogeneous of degree $\frac{q^2}{q-p}>0$. Hence scaling $w$ by a constant factor gives a solution of (\ref{2.1}) with $p$ replaced by $\frac{pq}{p-q}$, which satisfies $w_{\tau}(0)=w_{\tau}(\beta\Theta)=0$ and $w(\beta\Theta)=r^{\beta}w(0)$. Therefore, we obtain
\begin{equation}\label{4.5}
    \Theta\left(\frac{pq}{p-q},q,r^{\frac{q-p}{q}}\right)=\frac{q-p}{q}\Theta(p,q,r).
\end{equation}

To show the duality relation (\ref{4.3}), define
\begin{equation}\label{4.6}
    \tilde{u}(\eta)=\max\limits_{\substack{\tilde{\theta}\in[0,\Theta]\\ |\tilde{\theta}-\eta|\le\frac{\pi}{2}}}\frac{\cos(\tilde{\theta}-\eta)}{u(\tilde{\theta})},\quad\quad \eta\in [0,\Theta],
\end{equation}
where $u:[0,\Theta]\to\mathbb{R}$ is given as above. 
%%%%%%%%%%%%%% 

First, notice that at the critical point $\theta$ of $\frac{\cos(\tilde{\theta}-\eta)}{u(\tilde{\theta})}$ with $|\theta-\eta|<\frac{\pi}{2}$,
\begin{align*}
    \left(\frac{\cos(\theta-\eta)}{u(\theta)}\right)_{\theta}
    =\frac{-u\sin(\theta-\eta)-u_{\theta}\cos(\theta-\eta)}{u^2}=0,
\end{align*}
then it follows from $u_{\theta\theta}+u>0$ that
\begin{align*}
     \left(\frac{\cos(\theta-\eta)}{u(\theta)}\right)_{\theta\theta}
    =\frac{-(u_{\theta\theta}+u)\cos(\theta-\eta)}{u^2}<0,
\end{align*}
which implies that $\frac{\cos(\theta-\eta)}{u(\theta)}$ can have only maxima as critical points. If $\eta\in(0,\frac{\pi}{2})$, using $u_{\theta}(0)=0$, we have
\begin{align*}
    \left.\left(\frac{\cos(\tilde{\theta}-\eta)}{u(\tilde{\theta})}\right)_{\tilde{\theta}}\right|_{\tilde{\theta}=0}=\frac{\sin(\eta)}{u(0)}>0,
\end{align*}
thus the maximum of $\frac{\cos(\tilde{\theta}-\eta)}{u(\tilde{\theta})}$ is never achieved at $0$. Similarly, the maximum is never achieved at $\Theta$ if $\eta\in(\Theta-\frac{\pi}{2},\Theta)$. Therefore, the maximum of (\ref{4.6}) is only taken at the interior of $[0,\Theta]$ for $\eta\in(0,\Theta)$. Moreover, it is obvious that $ \tilde{u}(0)=\frac{1}{u(0)}$ and $\tilde{u}(\Theta)=\frac{1}{u(\Theta)}.$

Next suppose that $\theta=\theta(\eta)$ attains the maximum of (\ref{4.6}), we have
\begin{align}\label{4.7}
    u_{\theta}=u\tan(\eta-\theta),
\end{align}
so that
\begin{align*}
     u_{\theta}^2+u^2=\frac{1}{\tilde{u}^2},\quad
     u_{\theta}(u_{\theta\theta}+u)\frac{d\theta}{d\eta}=-\frac{\tilde{u}_{\eta}}{\Tilde{u}^3}.
\end{align*}
Since $u_{\theta\theta}+u>0$, differentiating both sides of (\ref{4.7}) yields
\begin{align*}
    \frac{d\theta}{d\eta}=\frac{u_{\theta}^2+u^2}{u(u_{\theta\theta}+u)}>0.
\end{align*}
Thus $\theta=\theta(\eta)$ is a monotone increasing function, and the following relations hold:
\begin{align}\label{4.8}
    \tilde{u}=(u_{\theta}^2+u^2)^{-\frac{1}{2}},\quad
    \frac{\tilde{u}_{\eta}}{\tilde{u}}=-\frac{u_{\theta}}{u},\quad
    \frac{d\eta}{d\theta}=\tilde{u}^2 u(u_{\theta\theta}+u).
\end{align}

On the other hand, fix $\theta_0\in(0,\Theta)$ and let $|\eta-\theta_0|<\frac{\pi}{2}$, we have
\begin{align*}
     \left(\frac{\cos(\eta-\theta_0)}{\tilde{u}(\eta)}\right)_{\eta}
    =\frac{-\tilde{u}\sin(\eta-\theta_0)-\tilde{u}_{\eta}\cos(\eta-\theta_0)}{{\tilde{u}}^2}=0
\end{align*}
if and only if
\begin{align*}
    \tan(\eta-\theta_0)=-\frac{\tilde{u}_{\eta}}{\tilde{u}}=\frac{u_{\theta}}{u}
    =\tan(\eta-\theta(\eta)),
\end{align*}
which implies that $\frac{\cos(\eta-\theta_0)}{\tilde{u}(\eta)}$ attains its critical value when $\theta(\eta)=\theta_0+k\pi,\ k\in\mathbb{Z}$.

It follows that 
\begin{equation}\label{4.9}
    \max\limits_{\substack{\tilde{\eta}\in[0,\Theta]\\ |\tilde{\eta}-\theta|\le\frac{\pi}{2}}}\frac{\cos(\tilde{\eta}-\theta)}{\tilde{u}(\tilde{\eta})}
    =\frac{\cos(\eta(\theta)-\theta)}{\tilde{u}(\eta(\theta))}
    =u(\theta).
\end{equation}
The same argument gives
\begin{align}\label{4.10}
    u=(\tilde{u}_{\eta}^2+\tilde{u}^2)^{-\frac{1}{2}},\quad
    \frac{u_{\theta}}{u}=-\frac{\tilde{u}_{\eta}}{\tilde{u}},\quad
    \frac{d\theta}{d\eta}=u^2 \tilde{u}(\tilde{u}_{\eta\eta}+\tilde{u}).
\end{align}

Using (\ref{4.8}) and (\ref{4.10}), we have
\begin{align*}
    (u_{\theta\theta}+u)(\tilde{u}_{\eta\eta}+\tilde{u})=u^{-3}\tilde{u}^{-3},
\end{align*}
then
\begin{align*}
    1 =u^{1-p}(u_{\theta}^2+u^2)^{\frac{q-2}{2}}(u_{\theta\theta}+u)
    =(\tilde{u}_{\eta}^2+\tilde{u}^2)^{-\frac{1-p}{2}}\tilde{u}^{2-q} (\tilde{u}_{\eta\eta}+\tilde{u})^{-1}u^{-3}\tilde{u}^{-3},
\end{align*}
which implies
\begin{align*}
    \tilde{u}^{1+q}(\tilde{u}_{\eta}^2+\tilde{u}^2)^{\frac{-p-2}{2}}(\tilde{u}_{\eta\eta}+\tilde{u})=1.
\end{align*}
Noting that $\frac{\tilde{u}(0)}{\tilde{u}(\Theta)}=\frac{u(\Theta)}{u(0)}=r$, we obtain
\begin{equation}\label{4.11}
    \Theta(-q,-p,r)=\Theta(p,q,r).
\end{equation}

It is apparent from (\ref{4.5}) and (\ref{4.11}) that the duality relation (\ref{4.2}) holds.
\end{proof}

In fact, we have a duality relation of the solution for the $L_p$ dual Minkowski problem with a prescribed function $f=1$ in \cite{CHZ19}.
\begin{thm}[\cite{CHZ19}]\label{Thm4.2}
If the support function $u$ of a smooth, 2-dimensional, strictly convex body $K$ satisfies the Monge-Ampère equation
\begin{equation*}
    u^{1-p}(u^2+|\nabla u|^2)^{\frac{q-2}{2}}\det(\nabla^2 u+u I)=1,
\end{equation*}
then the support function $\tilde{u}$ of the polar body $K^*$ satisfies the following equation
\begin{equation*}
    \tilde{u}^{1+q}(\tilde{u}^2+|\nabla\tilde{u}|^2)^{\frac{-p-2}{2}}\det(\nabla^2 \tilde{u}+\tilde{u}I)=1.
\end{equation*}
\end{thm}

By use of the duality relation (\ref{4.3}), we give special values and asymptotic behaviors for the case where $p<0$.

\begin{cor}\label{Cor4.3}
We have
\begin{align}\label{4.12}
    \lim\limits_{r\to \infty}\Theta(p,q,r)=\frac{p}{p-q}\cdot\frac{\pi}{2},\quad\quad\quad p<q\le 0.
\end{align}

When $r$ tends to $1$, we have
\begin{align}\label{4.13}
    & \Theta(p,q,r)=\frac{\pi}{\sqrt{q-p}}+\frac{\pi}{96\sqrt{q-p}}(q^2-pq+p^2+3p-3q)(r-1)^2+o\left((r-1)^2\right),\quad p<0.
\end{align}
\end{cor}

\section{Monotonicity in $p$ and $q$}
\label{sec:5}

In this section, we will prove that $\Theta(p,q,r)$ is monotone in $p$ and $q$. More specifically, we have the following theorem.

\begin{thm}\label{Thm5.1}
$\Theta(p,q,r)$ is monotone increasing in $p$ and is monotone decreasing in $q$.
\end{thm}

\subsection{Monotonicity in $p$}$\ $

To show $\Theta(p,q,r)$ is monotone increasing in $p$, denote
\begin{align*}
    I(p,q,r,x)=\left(\frac{r^p-r^q}{r^p-1}+\frac{r^q-1}{r^p-1}x^p\right)^{\frac{2}{q}}-x^2,
\end{align*}
and observe that if either $p<\bar{p}<0$, or $0<p<\bar{p}$, or $p<0<\bar{p}$ with $q>0$, we have
\begin{align*}
    Q(x)&:=(I(p,q,r,x)+x^2)^{\frac{q}{2}}-(I(\bar{p},q,r,x)+x^2)^{\frac{q}{2}}\\
    &=\frac{r^p-r^q}{r^p-1}+\frac{r^q-1}{r^p-1}x^p-\frac{r^{\bar{p}}-r^q}{r^{\bar{p}}-1}-\frac{r^q-1}{r^{\bar{p}}-1}x^{\bar{p}},
\end{align*}
this is zero for $x=1$ and for $x=r$. A direct computation gives
\begin{align*}
    \frac{\partial}{\partial x}\left(x^{1-p}\frac{\partial Q}{\partial x}\right)
    =-\bar{p}(\bar{p}-p)x^{\bar{p}-p-1}\frac{r^q-1}{r^{\bar{p}}-1}<0,
\end{align*}
from which it follows that $Q(x)$ can have only maxima as critical points in $x$, and therefore that $Q(x)>0$, or equivalently, $I(p,q,r,x)>I(\bar{p},q,r,x)$ for all $x\in (1,r)$. Then
\begin{align*}
    \Theta(p,q,r)-\Theta(\bar{p},q,r)=\int_1^r\left(\frac{1}{\sqrt{I(p,q,r,x)}}-\frac{1}{\sqrt{I(\bar{p},q,r,x)}}\right)dx<0.
\end{align*}

If $q<0$, a similar argument shows that $Q(x)<0$, equivalently, $I(p,q,r,x)>I(\bar{p},q,r,x)$ for all $x\in (1,r)$. Thus $\Theta(p,q,r)$ is also increasing in $p$ for $q<0$.

If $q=0$, denote
\begin{align*}
    Q(x)&:=\frac{\log\left(I(p,0,r,x)+x^2\right)}{\log r}-\frac{\log\left(I(\bar{p},0,r,x)+x^2\right)}{\log r}
    =\frac{2(x^p-1)}{r^p-1}-\frac{2(x^{\bar{p}}-1)}{r^{\bar{p}}-1},
\end{align*}
a similar argument shows that $Q(x)>0$, equivalently, $I(p,q,r,x)>I(\bar{p},q,r,x)$ for all $x\in (1,r)$. Thus $\Theta(p,0,r)$ is increasing in $p$.

\subsection{Monotonicity in $q$}$\ $

Actually, the monotonicity in $q$ of $\Theta(p,q,r)$ can be deduced from the monotonicity in $p$ and the duality (\ref{4.3}). For completeness, we gives a direct proof of the monotonicity in $q$.

If $0<q<\bar{q}$, we have
\begin{align*}
    P(x)&:=(I(p,q,r,x)+x^2)^{\frac{q}{2}}-(I(p,\bar{q},r,x)+x^2)^{\frac{q}{2}}\\
    &=\frac{r^p-r^q}{r^p-1}+\frac{r^q-1}{r^p-1}x^p-\left(\frac{r^{p}-r^{\bar{q}}}{r^{p}-1}+\frac{r^{\bar{q}}-1}{r^{p}-1}x^p\right)^{\frac{q}{\bar{q}}},
\end{align*}
this is zero for $x=1$ and for $x=r$. A direct computation gives
\begin{align*}
    \frac{\partial}{\partial x}\left(x^{1-p}\frac{\partial P}{\partial x}\right)
    =-p^2 x^{p-1}\cdot\frac{q}{\bar{q}}\left(\frac{q}{\bar{q}}-1\right)
    \left(\frac{r^{p}-r^{\bar{q}}}{r^{p}-1}+\frac{r^{\bar{q}}-1}{r^{p}-1}x^p\right)^{\frac{q}{\bar{q}}-2}\left(\frac{r^{\bar{q}}-1}{r^{p}-1}\right)^2 >0,
\end{align*}
from which it follows that $P(x)$ can have only minima as critical points in $x$, and therefore that $P(x)<0$, or equivalently, $I(p,q,r,x)<I(p,\bar{q},r,x)$ for all $x\in (1,r)$. Then
\begin{align*}
    \Theta(p,q,r)-\Theta(p,\bar{q},r)=\int_1^r\left(\frac{1}{\sqrt{I(p,q,r,x)}}-\frac{1}{\sqrt{I(p,\bar{q},r,x)}}\right)dx>0.
\end{align*}
In particular, when $p=0$ we denote $P(x)=(I(0,q,r,x)+x^2)^{\frac{q}{2}}-(I(0,\bar{q},r,x)+x^2)^{\frac{q}{2}}$, then the monotonicity of $\Theta(0,q,r)$ in $q$ follows in a similar manner.

If either $q<\bar{q}<0$ or $q<0<\bar{q}$, a similar argument shows that $P(x)>0$, equivalently, $I(p,q,r,x)<I(p,\bar{q},r,x)$ for all $x\in (1,r)$. Thus $\Theta(p,q,r)$ is also decreasing in $q$.

If either $q<\bar{q}=0$ or $q=0<\bar{q}$, since $\Theta(p,q,r)$ is continuous in $q$, then the monotonicity of $\Theta(p,q,r)$ in $q$ holds in this case.

\section{Monotonicity in $r$}
\label{sec:6}

Due to the asymptotic behavior (\ref{3.10}) when $r$ is close to $1$, we choose the following transformation:
\begin{align}\label{6.1}
    x^{\beta_0}=v=\frac{r^{\beta_0}-1}{2}z+\frac{r^{\beta_0}+1}{2},\quad z\in (-1,1),
\end{align}
where $\beta_0=\frac{2q-p}{3}$. Then we obtain
\begin{align}\label{6.2}
    \Theta(p,q,r)=\int_{-1}^1\frac{dz}{\sqrt{J(p,q,r,z)}},
\end{align}
where $J(p,q,r,z)=\tilde{I}(p,q,r,\beta_0,z)$, i.e.
\begin{equation}\label{6.3}
\begin{split}
    & J(p,q,r,z)=\frac{4\beta_0^2}{(r^{\beta_0}-1)^2}\left[v^{\frac{2(\beta_0-1)}{\beta_0}}\left(\frac{r^p-r^q}{r^p-1}+\frac{r^q-1}{r^p-1}v^{\frac{p}{\beta_0}}\right)^{\frac{2}{q}}-v^2\right],\quad p\neq0,\\
%    & J(p,q,r,z)=\frac{4\beta_0^2}{(r^{\beta_0}-1)^2}\left[\left(\frac{r^p-r^q}{r^p-1}v^{\frac{q(\beta_0-1)}{\beta_0}}+\frac{r^q-1}{r^p-1}v^{\frac{p+q(\beta_0-1)}{\beta_0}}\right)^{\frac{2}{q}}-v^2\right],\quad p\neq0\\
    & J(0,q,r,z)=\frac{4\beta_0^2}{(r^{\beta_0}-1)^2}\left[v^{\frac{2(\beta_0-1)}{\beta_0}}\left(1+\frac{r^q-1}{\beta_0\log r}\log v\right)^{\frac{2}{q}}-v^2\right].
\end{split}
\end{equation}

Using the duality relation (\ref{4.3}), we may assume $q>0$, then $\beta_0>0$. Let $t=\frac{r^{\beta_0}+1}{r^{\beta_0}-1}\in(1,+\infty)$ and $z\in (-1,1)$, define
\begin{align}\label{6.4}
    \tilde{J}(p,q,t,z)=\frac{1}{\beta_0^2} J(p,q,r,z).
\end{align}

Our purpose is to show that $\tilde{J}(p,q,t,z)$ is monotone in $t$ for each $z$.

\subsection{The case $p\neq0$}$\ $

If $p\neq0$, assume $\rho=\frac{p}{q}$. Note that
\begin{align*}
    r=\left(\frac{t+1}{t-1}\right)^{\frac{1}{\beta_0}},\quad v=\frac{t+z}{t-1},
\end{align*}
then we have
\begin{equation}\label{6.5}
    \tilde{J}(p,q,t,z)=(z+t)^2\left[K_{\rho}(t,z)^{\frac{2}{q}}-1\right],
\end{equation}
where
\begin{align*}
    K_{\rho}(t,z)
    &=\frac{\left(\frac{t+1}{t-1}\right)^{\frac{p}{\beta_0}}-\left(\frac{t+1}{t-1}\right)^{\frac{q}{\beta_0}}}{\left(\frac{t+1}{t-1}\right)^{\frac{p}{\beta_0}}-1}\left(\frac{t+z}{t-1}\right)^{\frac{-q}{\beta_0}}
    +\frac{\left(\frac{t+1}{t-1}\right)^{\frac{q}{\beta_0}}-1}{\left(\frac{t+1}{t-1}\right)^{\frac{p}{\beta_0}}-1}\left(\frac{t+z}{t-1}\right)^{\frac{p-q}{\beta_0}}\\
    &=\frac{A^{1-\rho}-B^{1-\rho}}{A^{-\rho}-B^{-\rho}}+\frac{A-B}{A^{\rho}-B^{\rho}},\quad A=\left(\frac{t+1}{t+z}\right)^{\frac{3}{2-\rho}},\  B=\left(\frac{t-1}{t+z}\right)^{\frac{3}{2-\rho}}.
\end{align*}

Note that
\begin{align*}
    \frac{\partial\tilde{J}}{\partial t}
    =(z+t)^2\left[\frac{2}{z+t}\left(K_{\rho}^{\frac{2}{q}}-1\right)+\frac{2}{q}K_{\rho}^{\frac{2}{q}-1}\frac{\partial K_{\rho}}{\partial t}\right]>0
\end{align*}
is equivalent to
\begin{equation}\label{6.6}
    \frac{\partial K_{\rho}}{\partial t}>\frac{q}{z+t}\left(K_{\rho}^{1-\frac{2}{q}}-K_{\rho}\right).
\end{equation}

At first we consider the case where $\rho=-1$. Then
\begin{align*}
    K_{-1}(t,z)=\frac{t^2+2tz+1}{(z+t)^2}.
\end{align*}
A direct computation gives rise to
\begin{align}\label{6.7}
    \frac{\partial K_{-1}}{\partial t}=\frac{2(z^2-1)}{(z+t)^3}=\frac{2}{z+t}\left(1-K_{-1}\right),
\end{align}
which implies
\begin{align*}
    \frac{\partial\tilde{J}}{\partial t}(-2,2,t,z)\equiv0,\quad t>1,\ |z|<1.
\end{align*}
Hence we get
\begin{align*}
    \frac{\partial\Theta}{\partial r}(-2,2,r)
    =-\int_{-1}^1\frac{1}{\beta_0}\frac{\partial t}{\partial r}\frac{\partial \tilde{J}}{\partial t}\frac{dz}{2\tilde{J}^{\frac{3}{2}}}=0,
\end{align*}
and then
\begin{align*}
    \Theta(-2,2,r)\equiv\lim\limits_{r\to 1}\Theta(-2,2,r)=\frac{\pi}{2}.
\end{align*}

Furthermore, we have

\begin{thm}\label{Thm6.1}
For each $t>1$ and $z\in(-1,1)$, then
\begin{align}\label{6.8}
    \frac{\partial\tilde{J}}{\partial t}(-q,q,t,z)
    \left\{\begin{array}{l}
    <0,\quad 0<q<2\\
    >0,\quad q>2. \end{array}\right. 
\end{align}

Moreover, for each $r>1$, then
\begin{align}\label{6.9}
    \frac{\partial\Theta}{\partial r}(-q,q,r)
    \left\{\begin{array}{l}
    <0,\quad 0<q<2\\
    >0,\quad q>2. \end{array}\right. 
\end{align}
\end{thm}
\begin{proof}
It follows from (\ref{6.7}) and weighted average inequalities that
\begin{align*}
     \frac{\partial K_{-1}}{\partial t}
     =\frac{2}{z+t}\left(1-K_{-1}\right)
     <\frac{q}{z+t}\left(K_{-1}^{1-\frac{2}{q}}-K_{-1}\right),\quad 0<q<2,
\end{align*}
and
\begin{align*}
    \frac{\partial K_{-1}}{\partial t}
    =\frac{2}{z+t}\left(1-K_{-1}\right)
    >\frac{q}{z+t}\left(K_{-1}^{1-\frac{2}{q}}-K_{-1}\right),\quad q>2.
\end{align*}

Using the characterization (\ref{6.6}), we obtain (\ref{6.8}). Thus
\begin{align*}
    \frac{\partial\Theta}{\partial r}(-q,q,r)
    =-\int_{-1}^1\frac{1}{\beta_0}\frac{\partial t}{\partial r}\frac{\partial \tilde{J}}{\partial t}\frac{dz}{2\tilde{J}^{\frac{3}{2}}}
    \left\{\begin{array}{l}
    <0,\quad 0<q<2\\
    >0,\quad q>2. \end{array}\right.
\end{align*}
\end{proof}

The general case will be discussed later.

\subsection{The case $p=0$}$\ $

If $p=0$, we have
\begin{align}\label{6.10}
    \tilde{J}(0,q,t,z)=(z+t)^2\left[L(t,z)^{\frac{2}{q}}-1\right],
\end{align}
where
\begin{align*}
    L(t,z)
    &=\left(\frac{t+z}{t-1}\right)^{-\frac{3}{2}}\left[1+\frac{\left(\frac{t+1}{t-1}\right)^{\frac{3}{2}}-1}{\log\left(\frac{t+1}{t-1}\right)}\log\left(\frac{t+z}{t-1}\right)\right]\\
    &=\frac{B\log A-A\log B}{\log A-\log B},\quad A=\left(\frac{t+1}{t+z}\right)^{\frac{3}{2}},\ B=\left(\frac{t-1}{t+z}\right)^{\frac{3}{2}}.
\end{align*}

Note that
\begin{align*}
    \frac{\partial\tilde{J}}{\partial t}(0,q,t,z)
    =(z+t)^2\left[\frac{2}{z+t}\left(L^{\frac{2}{q}}-1\right)+\frac{2}{q}L^{\frac{2}{q}-1}\frac{\partial L}{\partial t}\right]>0
\end{align*}
is equivalent to
\begin{equation}\label{6.11}
    \frac{\partial L}{\partial t}>\frac{q}{z+t}\left(L^{1-\frac{2}{q}}-L\right).
\end{equation}

\subsection{Monotonicity in $r$}$\ $

The following important lemma was proved by Ben Andrews \cite{Ben03}, which completely solves the case where $q=2$.

\begin{lem}\label{Lem6.2}
If $q=2$, for each $t>1$ and $z\in(-1,1)$, then
\begin{align}\label{6.12}
    \frac{\partial\tilde{J}}{\partial t}(p,2,t,z)
    \left\{\begin{array}{l}
    <0,\quad -2<p\le 0\\
    >0,\quad p<-2. \end{array}\right. 
\end{align}

Moreover, for each $r>1$, then
\begin{align}\label{6.13}
    \frac{\partial\Theta}{\partial r}(p,2,r)\left\{\begin{array}{l}
    <0,\quad -2<p\le 0\\
    >0,\quad p<-2. \end{array}\right. 
\end{align}
\end{lem}

Using the characterizations (\ref{6.6}) and (\ref{6.11}) and Lemma \ref{Lem6.2}, we have
\begin{thm}\label{Thm6.3}
For each $t>1$ and $z\in(-1,1)$, then
\begin{align}\label{6.14}
    \frac{\partial\tilde{J}}{\partial t}(p,q,t,z)
    \left\{\begin{array}{l}
    <0,\quad 0\le-p<q\le 2\\
    >0,\quad -p>q\ge 2. \end{array}\right. 
\end{align}

Moreover, for each $r>1$, then
\begin{align}\label{6.15}
    \frac{\partial\Theta}{\partial r}(p,q,r)
    \left\{\begin{array}{l}
    <0,\quad 0\le-p<q\le 2\\
    >0,\quad -p>q\ge 2. \end{array}\right.
\end{align}
\end{thm}
\begin{proof}
By substituting $p=2\rho$ into Lemma \ref{Lem6.2}, we obtain
\begin{align*}
    & \frac{\partial K_{\rho}}{\partial t}<\frac{2}{z+t}\left(1-K_{\rho}\right),\quad -1<\rho<0,\\
    & \frac{\partial K_{\rho}}{\partial t}>\frac{2}{z+t}\left(1-K_{\rho}\right),\quad \rho<-1.
\end{align*}

Using weighted average inequalities, we conclude that
\begin{align*}
     & \frac{\partial K_{\rho}}{\partial t}<\frac{q}{z+t}\left(K_{\rho}^{1-\frac{2}{q}}-K_{\rho}\right),\quad -1<\rho<0,\ 0<q\le 2,\\
    & \frac{\partial K_{\rho}}{\partial t}>\frac{q}{z+t}\left(K_{\rho}^{1-\frac{2}{q}}-K_{\rho}\right),\quad \rho<-1,\ q\ge 2,
\end{align*}
and
\begin{align*}
     & \frac{\partial L}{\partial t}<\frac{q}{z+t}\left(L^{1-\frac{2}{q}}-L\right),\quad 0<q\le 2,
\end{align*}
which imply (\ref{6.14}). Thus
\begin{align*}
    \frac{\partial\Theta}{\partial r}(p,q,r)
    =-\int_{-1}^1\frac{1}{\beta_0}\frac{\partial t}{\partial r}\frac{\partial \tilde{J}}{\partial t}\frac{dz}{2\tilde{J}^{\frac{3}{2}}}
    \left\{\begin{array}{l}
    <0,\quad 0\le -p<q\le 2\\
    >0,\quad -p>q\ge 2\end{array}\right.
\end{align*}
for each $r>1$.
\end{proof}

According to Theorem \ref{Thm4.1}, we can generalize the monotonicity results in Theorem \ref{Thm6.1} and Theorem \ref{Thm6.3}.

\begin{thm}\label{Thm6.4}
$\Theta(p,q,r)$ is monotone in $r$ if $(p,q)$ satisfies the conditions as follows.
\begin{enumerate}[(i)]
    \item If $p\le-2,\ q\ge2$ and $(p,q)\neq (-2,2)$, then
    \begin{align*}
     \frac{\partial\Theta}{\partial r}(p,q,r)>0,\quad\forall\ r>1. 
     \end{align*}
     
     \item If $q>p,\ q\ge2$, $p\ge\frac{2q}{2+q}$ and $(p,q)\neq (1,2)$, then
     \begin{align*}
     \frac{\partial\Theta}{\partial r}(p,q,r)>0,\quad\forall\ r>1. 
     \end{align*}
     
     \item If $q>p,\ p\le -2$, $q\le \frac{2p}{2-p}$ and $(p,q)\neq (-2,-1)$, then
     \begin{align*}
     \frac{\partial\Theta}{\partial r}(p,q,r)>0,\quad\forall\ r>1. 
     \end{align*}
     
     \item If $q\ge\frac{2p}{2-p},\ p\ge-2,\ q\le 2$, and $(p,q)\neq (0,0),\ (1,2),\ (-2,-1)$, or $(-2,2)$, then
     \begin{align*}
     \frac{\partial\Theta}{\partial r}(p,q,r)<0,\quad\forall\ r>1. 
     \end{align*}
\end{enumerate}
\end{thm}

\begin{proof}
Due to Theorem \ref{Thm6.3}, it follows that
\begin{align*}
    \frac{\partial\Theta}{\partial r}(p,q,r)
    <0,\quad 0\le-p<q\le 2.
\end{align*}
Using the duality relation (\ref{4.3}), we have
\begin{align*}
    & \frac{\partial\Theta}{\partial r}(p,q,r)
    =\frac{\partial\Theta}{\partial r}(-q,-p,r)
    <0,\quad 0\le q<-p\le 2.
\end{align*}
Combining Theorem \ref{Thm6.1}, thus
\begin{align*}
     \frac{\partial\Theta}{\partial r}(p,q,r)
    <0,\quad -2\le p\le 0\le q\le 2,\ (p,q)\neq(0,0),(-2,2).
\end{align*}

By the duality relations (\ref{4.1}) and (\ref{4.2}), we have
\begin{align*}
     \frac{\partial\Theta}{\partial r}(p,q,r)
    =r^{-\frac{p}{q}}\frac{\partial\Theta}{\partial r}\left(\frac{pq}{p-q},q,r^{\frac{q-p}{q}}\right)
    <0,\quad 0<p<1,\ \frac{2p}{2-p}\le q\le 2,
\end{align*}
and
\begin{align*}
     \frac{\partial\Theta}{\partial r}(p,q,r)
    =r^{-\frac{q}{p}}\frac{\partial\Theta}{\partial r}\left(p,\frac{pq}{q-p},r^{\frac{p-q}{p}}\right)
    <0,\quad -1<q<0,\ \frac{2q}{2+q}\ge p\ge -2.
\end{align*}
This completes the proof of the case (iv). The proof of the rest parts follows in a similar manner.
\end{proof}

\begin{figure}[htbp]
\centering
\begin{tikzpicture}[scale=0.8][>=Stealth] 
%\draw[->](-8,0)--(8.3,0) node[right]{$p$-axis};
%\draw[->](0,-4)--(0,12.3) node[above]{$q$-axis};
\filldraw[opacity=0.6, draw=white, fill=blue!!100] (-2,0)--(0,0)--(0,2)--(-2,2)--cycle;
\filldraw[opacity=0.6, draw=white, fill=blue!!100, domain=0:1] plot(\x,{2*\x/(2-\x)})--(0,2)--cycle;
\filldraw[opacity=0.6, draw=white, fill=blue!!100, domain=-2:0] plot(\x,{2*\x/(2-\x)})--(-2,0)--cycle;
%%%%
\filldraw[opacity=0.8, draw=white, fill=red!!100] (-2,2)--(-2,8)--(-8,8)--(-8,2)--cycle;
\filldraw[opacity=0.8, draw=white, fill=red!!100] (2,2)--(8,8)--(2,8)--cycle;
\filldraw[opacity=0.8, draw=white, fill=red!!100] (-2,-2)--(-8,-2)--(-8,-4)--(-4,-4)--cycle;
\filldraw[opacity=0.8, draw=white, fill=red!!100, domain=1:1.6] plot(\x,{2*\x/(2-\x)})--(2,8)--(2,2)--cycle;
\filldraw[opacity=0.8, draw=white, fill=red!!100, domain=-8:-2] plot(\x,{2*\x/(2-\x)})--(-2,-2)--(-8,-2)--cycle;
%number=infinity
\filldraw[opacity=0.8, fill=black!80] (-2,2) circle (0.1);
\filldraw[opacity=0.8, fill=black!80] (1,2) circle (0.1);
\filldraw[opacity=0.8, fill=black!80] (-2,-1) circle (0.1);

\draw[->](-8,0)--(8.3,0) node[right]{$p$-axis};
\draw[->](0,-4)--(0,8.3) node[above]{$q$-axis};
\foreach \x in {-7, -6, -5, -4, -3, -2, -1, 1, 2, 3, 4, 5, 6, 7} \draw (\x, 1pt) -- (\x, -1pt) node[anchor=north] {$\x$};
\foreach \y in {-3, -2, -1, 1, 2, 3, 4, 5, 6, 7} \draw (1pt, \y) -- (-1pt, \y) node[anchor=east] {$\y$};
\node[below right] at (0,0) {$0$};
\draw [densely dashed] (-8,2)--(8,2);
\draw [densely dashed] (-8,1)--(8,1);
\draw [densely dashed] (-2,-4)--(-2,8);
\draw [densely dashed] (-1,-4)--(-1,8);
\draw [densely dashed] (2,-4)--(2,8);
\draw [densely dashed] (-8,-2)--(8,-2);
\draw [densely dashed] (-8,8)--(4,-4);
\node[above right] at (-7,7) {$q=-p$};
\draw [densely dashed] (-4,-4)--(8,8);
\node[above right] at (7.5,7) {$q=p$};
\draw [densely dashed] (-8,-4)--(4,8);
\node[above right] at (3.5,7) {$q-p=4$};
\draw [densely dashed] [domain = -8: 1.6] plot ({\x},{2*\x/(2-\x)});
\node[above ] at (1,6) {$q=\frac{2p}{2-p}$};
\draw[thick, rotate=-135] (0,-1.414) ellipse (2.449 and 1.414);
\node[above] at (-6.5,4.5) {Case (i)};
\node[above] at (3.5,4.5) {Case (ii)};
\node[above] at (-6.5,-3) {Case (iii)};
\node[above] at (-1,1) {Case (iv)};
%\draw[domain = 0: 6.29] plot ({4*cos(\x r)*cos(\x r)-3},{2*cos(\x r)*cos(\x r)+3*sin(\x r)*cos(\x r)});
\end{tikzpicture}
\caption{Domains where $\Theta(p,q,r)$ is monotone in $r$\\
Blue domain: $\frac{\partial\Theta}{\partial r}<0$. Red domain: $\frac{\partial\Theta}{\partial r}>0.$}
\label{fig-2}
\end{figure}
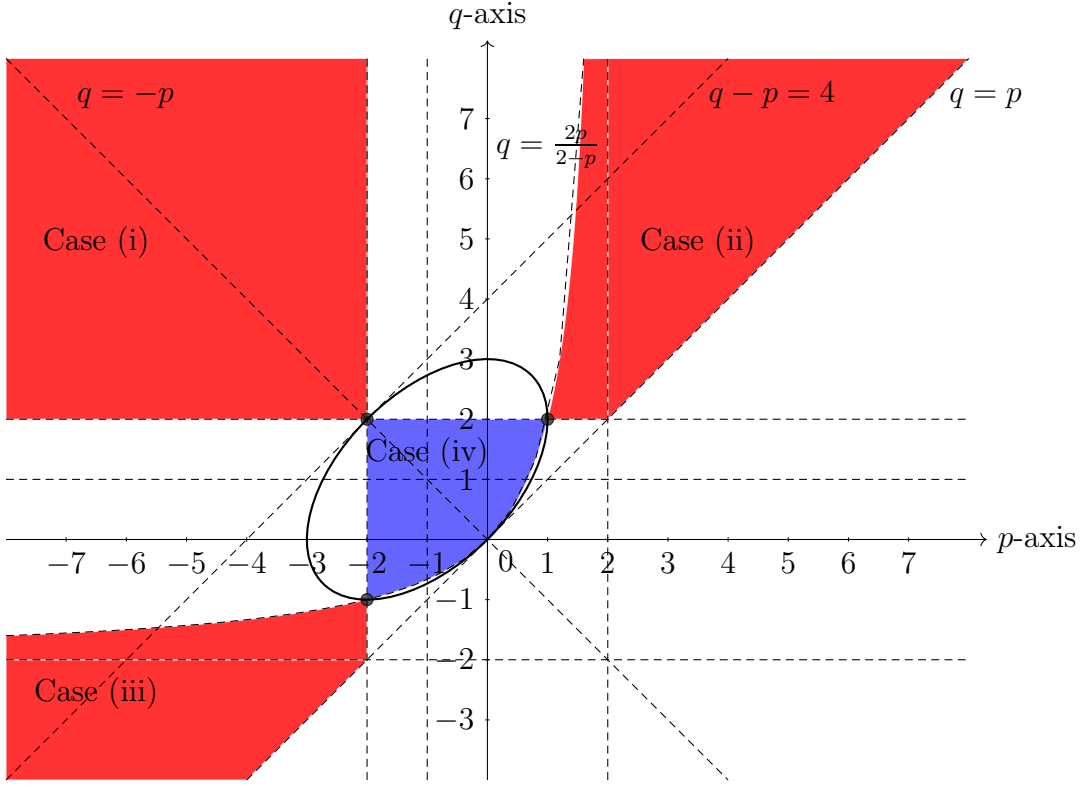

\begin{rem}
Consider the ellipse $\Phi$ defined by
\begin{equation}\label{6.16}
    p^2-pq+q^2+3p-3q=0.
\end{equation}

Recall that $\Theta(p,q,r)$ has the asymptotic expansions (\ref{3.14}) and (\ref{4.13}) as $r$ tends to $1$, we deduce that if $\Theta(p,q,r)$ monotonically increases in $r$ for any $r>1$, then $(p,q)$ must lie inside $\Phi$; if $\Theta(p,q,r)$ monotonically decreases in $r$ for any $r>1$, then $(p,q)$ must lie outside $\Phi$. This observation is consistent with the results in Theorem \ref{Thm6.4}, see Figure \ref{fig-2}.

On the other hand, according to the asymptotic behaviors in Theorem \ref{Thm3.2}, it is clear that there exist many $(p,q)$ such that $\Theta(p,q,r)$ is not monotone in $r$ for any $r>1$. For instance, when $p=0$ and $q>3$, the function $\Theta(p,q,r)$ is monotonically increasing in $r$ as $r$ tends to $1$, and monotonically decreasing in $r$ as $r$ tends to infinity.
\end{rem}

\section{Proofs of Theorem \ref{Thm1.2} and Theorem \ref{Thm1.3}}
\label{sec:7}

Recall that an embedded solution has total curvature $2\pi$, thus the period $\Theta(p,q,r)$ of a non-constant embedded solution must equal $\frac{\pi}{k}$ for some integer $k$.

We divide our proof of Theorem \ref{Thm1.2} in five situations.

\iffalse
Moreover, the four-vertex theorem states that an embedded closed curve must have at least four vertices, then for any embedded solution $u(\theta)$ of Eq.(1.1),
\begin{align*}
    0=(\log\kappa)_{\theta}=\frac{u_{\theta}}{u} \left[(1-p)+(q-2)u^p (u_{\theta}^2+u^2)^{-\frac{q}{2}}\right],
\end{align*}
hence its vertex satisfies $u_{\theta}=0$, or
\begin{align}
    (q-2)u^p (u_{\theta}^2+u^2)^{-\frac{q}{2}}=p-1.
\end{align}

Therefore, $r>1$ corresponds a non-constant embedded solution of Eq.(1.1), if and only if, the period $\Theta(p,q,r)$ satisfies the following condition:
\begin{enumerate}[(i)]
    \item For $p\le 1,\ q\ge 2$, or $1\le p<q\le 2$, then $\Theta(p,q,r)=\frac{\pi}{k}$ for some integer $k\ge2$.
    
    \item For $p<1,\ p<q<2$, or $q>2,\ 1<p<q$, then $\Theta(p,q,r)=\frac{\pi}{k}$ for some integer $k\ge1$. And if $k=1$, there exists $u(\theta_0)$ satisfies (7.1).
\end{enumerate}
\fi

\subsection{Monotone case}$\ $

First we study the embedded solutions of Eq.(\ref{1.2}) for the case where $\Theta(p,q,r)$ is monotone in $r$. According to Theorem \ref{Thm3.1}, Corollary \ref{Cor4.3} and Theorem \ref{Thm6.4}, it follows that if $(p,q)$ satisfies the conditions in Theorem \ref{Thm6.4}, the values of $\Theta(p,q,r)$ lie between 
\begin{align}\label{7.1}
    \lim\limits_{r\to 1}\Theta(p,q,r)=\frac{\pi}{\sqrt{q-p}},
\end{align}
and
\begin{align}\label{7.2}
    \lim\limits_{r\to \infty}\Theta(p,q,r)
=\left\{\begin{array}{l}
\frac{q}{q-p}\frac{\pi}{2},\quad \text{if}\ 0\le p<q\\
\frac{\pi}{2},\quad\quad\ \text{if}\ p<0<q\\ 
\frac{p}{p-q}\frac{\pi}{2},\quad \text{if}\ p<q\le 0
\end{array}\right.
=\frac{\pi}{\Xi(p,q)}.
\end{align}

If $p\le-2,\ q\ge2$ and $(p,q)\neq (-2,2)$, it follows from (\ref{7.1}), (\ref{7.2}) and the case (i) in Theorem \ref{Thm6.4} that $\Theta(p,q,r)$ monotonically increases in $r$ from $\frac{\pi}{\sqrt{q-p}}$ to $\frac{\pi}{2}$. If in addition $(k-1)^2<q-p\le k^2$ for some integer $k\ge3$, then $\Theta(p,q,r)$ equals $\frac{\pi}{m}$ for some $r>1$ if and only if 
$$2<m<\sqrt{q-p}.$$
Thus there exist exactly $(k-2)$ embedded solutions of Eq.(1.1), including the constant solution. Moreover, it is known that all solutions of Eq.(1.1) with $(p,q)=(-2,2)$ are given by
\begin{align}\label{7.3}
    u(\theta)=\sqrt{\lambda^2\cos^2(\theta-\theta_0)+\lambda^{-2}\sin^2(\theta-\theta_0)},
\end{align}
for $\lambda>0$ and $\theta_0\in[0,2\pi)$.

If $q>p,\ q\ge2$ and $p\ge\frac{2q}{2+q}$, it follows from (\ref{7.1}), (\ref{7.2}) and the case (ii) in Theorem \ref{Thm6.4} that $\Theta(p,q,r)$ monotonically increases in $r$ from $\frac{\pi}{\sqrt{q-p}}$ to $\frac{q}{q-p}\frac{\pi}{2}$. If in addition $(k-1)^2<q-p\le k^2$ for some integer $k$, then $\Theta(p,q,r)$ equals $\frac{\pi}{m}$ for some $r>1$ if and only if 
$$\frac{2(q-p)}{q}<m<\sqrt{q-p}.$$
Thus there exist exactly $(k-1)$ embedded solutions when $2p\le q$, and there exist exactly $k$ embedded solutions when $2p> q$. A similar argument holds for the case $q>p,\ p\le -2$ and $q\le \frac{2p}{2-p}$.

If $q\ge\frac{2p}{2-p},\ p\ge-2,\ q\le 2$ and $(p,q)\neq (0,0)$, it follows from (\ref{7.1}), (\ref{7.2}) and the case (iv) in Theorem \ref{Thm6.4} that $\Theta(p,q,r)$ monotonically decreases in $r$ from $\frac{\pi}{\sqrt{q-p}}$ to $\frac{\pi}{\Xi(p,q)}$. We now consider four subcases.
\begin{enumerate}[(i)]
    \item If in addition $0<q-p< 1$, then $\Theta(p,q,r)$ equals $\frac{\pi}{m}$ for some $r>1$ if and only if 
    $$\sqrt{q-p}<m<\Xi(p,q).$$
    Thus there exists a unique embedded solution when $2p\ge q$ or $p\ge 2q$, and there exist exactly two embedded solutions when $2p<q$ and $p<2q$. 
    % Note that the non-constant solution is not $\pi$-periodic.
    
    \item If in addition $1\le q-p\le 4$ and $(p,q)\neq(1,2),\ (-2,-1),\ (-2,2)$, we have
    \begin{align*}
        \frac{\pi}{2}\le \frac{\pi}{\Xi(p,q)}<\Theta(p,q,r)
        <\frac{\pi}{\sqrt{q-p}}\le\pi.
    \end{align*}
    Then $\Theta(p,q,r)$ never equals $\frac{\pi}{m}$ for any integer $m$. Hence the only embedded solution is the constant solution. 
    
    \item If $(p,q)=(1,2)$, a direct computation shows that all embedded solutions of Eq.(\ref{1.2}) with $(p,q)=(1,2)$ are given by
    \begin{align}\label{7.4}
        u(\theta)=1+\lambda\cos(\theta-\theta_0),
    \end{align}
    for $\lambda\ge0$ and $\theta_0\in[0,2\pi)$.
    
    \item If $(p,q)=(-2,-1)$, using the duality relation in Theorem \ref{Thm4.2}, we see that the embedded solution of Eq.(\ref{1.2}) with $(p,q)=(-2,-1)$ can be derived from the positive embedded solution of Eq.(\ref{1.2}) with $(p,q)=(1,2)$. That is, all embedded solutions of Eq.(\ref{1.2}) with $(p,q)=(-2,-1)$ are given by
    \begin{align}\label{7.5}
        u(\theta)=\frac{\sqrt{1-\mu^2\sin^2(\theta-\theta_0)}-\mu\cos(\theta-\theta_0)}{1-\mu^2},
    \end{align}
    for $0\le\mu<1$ and $\theta_0\in[0,2\pi)$.
\end{enumerate}

To summarize what we have proved, we obtain
\begin{lem}\label{Lem7.1}
The classification of the embedded solutions of Eq.(\ref{1.2}) for the monotone case is as follows:
\begin{enumerate}
    \item[Case (2)] \quad $0<q-p\le 1$.
    \begin{enumerate}
        \item[Subcase $1^\circ$] If $(p,q)=(1,2)$, then each embedded solution has the form 
        \begin{align*}
        u(\theta)=1+\lambda\cos(\theta-\theta_0),
        \end{align*}
        for $\lambda\ge0$ and $\theta_0\in[0,2\pi)$.
        
        \item[Subcase $2^\circ$] If $(p,q)=(-2,-1)$, then each embedded solution has the form 
        \begin{align*}
        u(\theta)=\frac{\sqrt{1-\mu^2\sin^2(\theta-\theta_0)}-\mu\cos(\theta-\theta_0)}{1-\mu^2},
        \end{align*}
        for $0\le\mu<1$ and $\theta_0\in[0,2\pi)$.
 
        \item[Subcase $3^\circ$] If $q-p=1$ and $(p,q)\neq (1,2),\ (-2,-1)$, then the embedded solution is unique. 
        
        \item[Subcase $4.1^\circ$] If $q-p<1$, $q\le 2p$ and $q\ge2$, then the embedded solution is unique. 
        
        \item[Subcase $4.2^\circ$] If $q-p<1$, $q\le 2p$ and $p\le\frac{2q}{2+q}\ (0<q<2)$, then the embedded solution is unique. 
        
        \item[Subcase $5.1^\circ$] If $q-p<1,\ q>2p,\ 2q\le p$ and $p\le -2$, then the embedded solution is unique.
        
        \item[Subcase $5.2^\circ$] If $q-p<1,\ q>2p,\ 2q\le p$ and $q\ge\frac{2p}{2-p}\ (-2<p<0)$, then the embedded solution is unique. 
        
        \item[Subcase $6^\circ$] If $q-p<1,\ q>2p$ and $2q>p$, then there exist exactly two embedded solutions.
        
        $\ $
    \end{enumerate}
    
    \item[Case (3)] \quad $1<q-p\le 4$.
    \begin{enumerate}
        \item[Subcase $1^\circ$] If $(p,q)=(-2,2)$, then each embedded solution has the form 
        $$u(\theta)=\sqrt{\lambda^2\cos^2(\theta-\theta_0)+\lambda^{-2}\sin^2(\theta-\theta_0)},$$
        for $\lambda>0$ and $\theta_0\in[0,2\pi)$.
        
        \item[Subcase $2.1^\circ$] If $p\ge-2$, $q\le2$ and $1<q-p\le 3$, then the embedded solution is unique.
        
        \item[Subcase $2.2^\circ$] If $q\ge 2p,\ 2q\ge p,\ 1<q-p\le 3$ and $p\ge\frac{2q}{2+q}\ (q>2)$, then the embedded solution is unique.
        
        \item[Subcase $2.3^\circ$] If $q\ge 2p,\ 2q\ge p,\ 1<q-p\le 3$ and $q\ge\frac{2p}{2-p}\ (p<-2)$, then the embedded solution is unique.
        
        \item[Subcase $3.1^\circ$] If $p\ge-2$, $q\le2,\ 3<q-p\le 4$ and $(p,q)\neq(-2,2)$, then the embedded solution is unique.
        
        \item[Subcase $3.2^\circ$] If $q\ge 2p,\ 2q\ge p,\ 3<q-p\le 4,\ q>2$ and $p\ge\frac{2q}{2+q}$, then the embedded solution is unique.
        
        \item[Subcase $3.4^\circ$] If $q\ge 2p,\ 2q\ge p,\ 3<q-p\le 4,\ p<-2$ and $q\le \frac{2p}{2-p}$, then the embedded solution is unique.
        
        \item[Subcase $4^\circ$] If $q\ge 2p$ and $2q< p$, then there exist exactly two embedded solutions.
        
        \item[Subcase $5^\circ$] If $q<2p$, then there exist exactly two embedded solutions.

        $\ $
    \end{enumerate}

    \item[Case ($4$)] \quad $(k-1)^2<q-p\le k^2,\quad k\ge3$.
    \begin{enumerate}
        \item[Subcase $1^\circ$] If $q\ge 2$ and $p\le -2$, then there exist exactly $(k-2)$ embedded solutions.
        
        \item[Subcase $4^\circ$] If $q\ge 2,\ p\ge \frac{2q}{2+q},\ p>-2$ and $q\ge 2p$, then there exist exactly $(k-1)$ embedded solutions.
        
        \item[Subcase $5^\circ$] If $q\ge 2,\ p\ge \frac{2q}{2+q},\ p>-2$ and $q<2p$, then there exist exactly $k$ embedded solutions.
        
        \item[Subcase $8^\circ$] If $p\le -2,\ q\le \frac{2p}{2-p},\ q< 2$ and $2q\ge p$, then there exist exactly $(k-1)$ embedded solutions.
        
        \item[Subcase $9^\circ$] If $p\le -2,\ q\le \frac{2p}{2-p},\ q<2$ and $2q<p$, then there exist exactly $k$ embedded solutions.
        
    \end{enumerate}
\end{enumerate}
\end{lem}

$\ $

\subsection{Case (1): $q-p\le 0$}$\ $

If $q-p\le0$, applying the strong maximum principle, we obtain the uniqueness of solutions of Eq.(\ref{1.2}). (see, for example \cite{HZ18, CL21})

\begin{lem}\label{Lem7.2}
In the case (1), we have
\begin{enumerate}
    \item[Case (1)] \quad $q-p\le 0$.
    \begin{enumerate}
        \item[Subcase $1^\circ$] If $q<p$, then the embedded solution is unique.
        \item[Subcase $2^\circ$] If $q=p$, then the embedded solution is unique up to scaling.
    \end{enumerate}
    \end{enumerate}
\end{lem}
    
$\ $

\subsection{Case (2): $0<q-p\le 1$}$\ $

If $0<q-p\le 1$, we only need to show that there exists a unique solution of Eq.(\ref{1.2}) if $q>p>\frac{2q}{2+q}\ (0<q<2)$. Note that the case where $p<q<\frac{2p}{2-p}\ (-2<p<0)$ can be proved in the same way.

% $\Theta(p,q,r)$ is monotone decreasing in $q$.
According to Theorem \ref{Thm5.1} and Theorem \ref{Thm6.3}, we get
$$\Theta(p,q,r)>\Theta(p,2,r)\ge\lim\limits_{r\to 1}\Theta(p,2,r)=\frac{\pi}{\sqrt{2-p}}\ge\pi,\quad\quad 1\le p<2,$$
and
$$\Theta(p,q,r)>\Theta(p,2p,r)>\lim\limits_{r\to\infty}\Theta(p,2p,r)=\pi,\quad\quad 0<p<1.$$
Thus the only embedded solution is the constant solution.

\begin{lem}\label{Lem7.3}
In the Case (2), we have
\begin{enumerate}
    \item[Case (2)] \quad $0<q-p\le 1$.
    \begin{enumerate}
        \item[Subcase $4.3^\circ$] If $q-p<1,\ q\le 2p$ and $p>\frac{2q}{2+q}\ (0<q<2)$, then the embedded solution is unique. 
        
        \item[Subcase $5.3^\circ$] If $q-p<1,\ q>2p,\ 2q\le p$ and $q<\frac{2p}{2-p}\ (-2<p<0)$, then the embedded solution is unique. 
    \end{enumerate}
\end{enumerate}
\end{lem}
Combining Lemma \ref{Lem7.1}, this completes the proof of Case (2).

$\ $

\subsection{Case (3): $1<q-p\le 4$}$\ $

If $1<q-p\le 3$, we only need to consider the case where $q>2$ and $p<\frac{2q}{2+q}$ because of the duality relation (\ref{4.3}).

Let $q-p=3$, we will estimate upper bounds of
\begin{align*}
    I(p,q,r,x)=\left\{\begin{array}{l}
\left(1+\frac{r^q-1}{r^p-1}(x^p-1)\right)^{\frac{2}{q}}-x^2,\quad p\neq0\\
\left(1+\frac{r^q-1}{\log r}\log x\right)^{\frac{2}{q}}-x^2,\quad\quad p=0
\end{array}\right.
\end{align*}
for each $x\in(1,r)$, then we can give a lower bound of
\begin{align*}
    \Theta(p,q,r)=\int_{1}^r \frac{dx}{\sqrt{I(p,q,r,x)}}.
\end{align*}

\begin{lem}\label{Lem7.4}
Suppose $q-p=3,\ -1<p<2$ and $q>\frac{2p}{2-p}$. If $r>1$ satisfies 
\begin{align}\label{7.6}
    \frac{p(r^q-1)}{q(r^p-1)}\le r^2,
\end{align}
then we have
\begin{align}\label{7.7}
    I(p,q,r,x)< J_1(r,x)=r^2+1-\frac{r^2}{x^2}-x^2,\quad\forall\ x\in(1,r).
\end{align}

On the other hand, if $r>1$ satisfies 
\begin{align}\label{7.8}
    \frac{p(r^q-1)}{q(r^p-1)}> r^2,
\end{align}
then we have
\begin{align}\label{7.9}
    I(p,q,r,x)< J_2(r,x)=r^2-2r+2x-x^2,\quad\forall\ x\in(1,r).
\end{align}
\end{lem}

\begin{proof}
First, assume $-1<p<0$ and $r>1$ satisfies (\ref{7.6}). Let
\begin{align*}
    P_1(x)=J_1(r,x)-I(p,q,r,x)=r^2+1-\frac{r^2}{x^2}-\left(1+\frac{r^q-1}{r^p-1}(x^p-1)\right)^{\frac{2}{q}},
\end{align*}
we have $P_1(1)=P_1(r)=0$ and
\begin{align*}
    P_1'(x)=\frac{2r^2}{x^3}-\frac{2p(r^q-1)}{q(r^p-1)}x^{p-1}\left(1+\frac{r^q-1}{r^p-1}(x^p-1)\right)^{\frac{2}{q}-1}.
\end{align*}
Then $P_1'(x)>0$ if and only if
\begin{align*}
    Q_1(x)=1+\frac{r^q-1}{r^p-1}(x^p-1)-\left(\frac{p(r^q-1)}{q(r^p-1)r^2}x^{p+2}\right)^{\frac{q}{q-2}}>0.
\end{align*}

It follows from (\ref{7.6}) that $Q_1(1)\ge 0$. A direct computation yields $Q_1(r)<0$ and
\begin{align*}
    Q_1''(x)=\frac{r^q-1}{r^p-1}\cdot p(p-1)x^{p-2}-\left(\frac{p(r^q-1)}{q(r^p-1)r^2}\right)^{\frac{q}{q-2}}\frac{(p+2)q}{q-2}\left(\frac{(p+2)q}{q-2}-1\right)x^{\frac{(p+2)q}{q-2}-2} <0.
\end{align*}
Thus $P_1'(x)$ changes from positive to negative as $x$ increases from $1$ to $r$, from which it follows that $P_1(x)>0$ for all $x\in (1,r)$. This completes the proof of (\ref{7.7}) for $-1<p<0$. 

Next, assume $-1<p<0$ and $r>1$ satisfies (\ref{7.8}). Let
\begin{align*}
    P_2(x)=J_2(r,x)-I(p,q,r,x)=r^2-2r+2x-\left(1+\frac{r^q-1}{r^p-1}(x^p-1)\right)^{\frac{2}{q}},
\end{align*}
we have $P_2(1)=(r-1)^2>0,\ P_2(r)=0$ and
\begin{align*}
    P_2'(x)=2-\frac{2p(r^q-1)}{q(r^p-1)}x^{p-1}\left(1+\frac{r^q-1}{r^p-1}(x^p-1)\right)^{\frac{2}{q}-1}.
\end{align*}
Then $P_2'(x)>0$ if and only if
\begin{align*}
    Q_2(x)=1+\frac{r^q-1}{r^p-1}(x^p-1)-\left(\frac{p(r^q-1)}{q(r^p-1)}x^{p-1}\right)^{\frac{q}{q-2}}>0.
\end{align*}

It follows from (\ref{7.8}) that $Q_2(r)< 0$. A direct computation yields $Q_2(1)<0$ and
\begin{align*}
   &  Q_2'(x)=\frac{r^q-1}{r^p-1}\cdot p x^{p-1}-\left(\frac{p(r^q-1)}{q(r^p-1)}\right)^{\frac{q}{q-2}}\frac{(p-1)q}{q-2} x^{\frac{(p-1)q}{q-2}-1} >0.
%   &  Q_2''(x)=\frac{r^q-1}{r^p-1}\cdot p(p-1)x^{p-2}-\left(\frac{p(r^q-1)}{q(r^p-1)}\right)^{\frac{q}{q-2}}\frac{(p-1)q}{q-2}\left(\frac{(p-1)q}{q-2}-1\right)x^{\frac{(p-1)q}{q-2}-2} <0.
\end{align*}
Thus $P_2'(x)$ is negative for $x\in (1,r)$, from which it follows that $P_2(x)>0$ for all $x\in (1,r)$. This completes the proof of (\ref{7.9}) for $-1<p<0$. 

Note that $-1<p<2$ and $p+3>\frac{2p}{2-p}$ implies $-1<p<\frac{\sqrt{33}-3}{2}$. The proof of (\ref{7.7}) and (\ref{7.9}) for $0\le p<\frac{\sqrt{33}-3}{2}$ is quite similar and so is omitted.
\end{proof}

Applying Lemma \ref{Lem7.4}, we prove the statements for $1<q-p\le 3$.

\begin{lem}\label{Lem7.5}
In the Case (3), we have
\begin{enumerate}
     \item[Case (3)] \quad $1<q-p\le 4$.
    \begin{enumerate}
        \item[Subcase $2.4^\circ$] If $q\ge 2p$, $2q\ge p,\ 1<q-p\le 3$ and $p<\frac{2q}{2+q}\ (q>2)$, then the embedded solution is unique.
        
        \item[Subcase $2.5^\circ$] If $q\ge 2p$, $2q\ge p,\ 1<q-p\le 3$ and $q>\frac{2p}{2-p}\ (p<-2)$, then the embedded solution is unique.
    \end{enumerate}
\end{enumerate}
\end{lem}
\begin{proof}
Notice that
\begin{align*}
    \int_1^r \frac{dx}{\sqrt{J_i(r,x)}}=\frac{\pi}{2},\quad i=1,2,\quad \forall\ r>1.
\end{align*}

It follows from Lemma \ref{Lem7.4} and Theorem \ref{Thm5.1} that
\begin{align*}
    & \Theta(p,q,r)<\Theta(p,p+1,r)\le\pi, & -1<p\le 1,\\
    & \Theta(p,q,r)<\Theta(p,2p,r)\le\pi, & 1<p<2,
    %\quad \forall\ r>1.
\end{align*}
and
\begin{align*}
    \Theta(p,q,r)\ge\Theta(p,p+3,r)
    =\int_{1}^r \frac{dx}{\sqrt{I(p,p+3,r,x)}}>\frac{\pi}{2},%\quad \forall\ r>1.
\end{align*}
for each $r>1$. Thus the only embedded solution is the constant solution.
\end{proof}

$\ $

If $3<q-p\le 4$, we only need to consider the case where $q>2$ and $-2<p<\frac{2q}{2+q}$. Using Theorem \ref{Thm5.1}, we have
\begin{align*}
    & \Theta(p,q,r)<\Theta(p,p+1,r)\le\pi, & -1<p\le 1,\\
    & \Theta(p,q,r)<\Theta(p,2p,r)\le\pi, & 1<p<2,
    %\quad \forall\ r>1.
\end{align*}
and
\begin{align*}
    \Theta(p,q,r)>\Theta(-2,q,r)>\frac{\pi}{\sqrt{q+2}}\ge\frac{\pi}{2\sqrt{2}},
\end{align*}
for each $r>1$. Hence, the non-constant embedded solution of Eq.(\ref{1.2}) must be $\pi$-periodic, equivalently, it satisfies
$$\Theta(p,q,r)=\frac{\pi}{2}.$$

\begin{lem}\label{Lem7.6}
In the Case (3), we have
\begin{enumerate}
     \item[Case (3)] \quad $1<q-p\le 4$.
    \begin{enumerate}
        \item[Subcase $3.3^\circ$] If $q\ge 2p$, $2q\ge p,\ 3<q-p\le 4,\ q>2$ and $p<\frac{2q}{2+q}$, then the non-constant embedded solution must be $\pi$-periodic if it exists.
        
        \item[Subcase $3.5^\circ$] If $q\ge 2p$, $2q\ge p,\ 3<q-p\le 4,\ p<-2$ and $q> \frac{2p}{2-p}$, then the non-constant embedded solution must be $\pi$-periodic if it exists.
    \end{enumerate}
\end{enumerate}
\end{lem}

$\ $

\subsection{Case (4): $q-p>4$}$\ $

If $(k-1)^2<q-p\le k^2$ for some integer $k\ge 3$,  we may assume $p>-2, q\ge 2$ and $p<\frac{2q}{2+q}$ because of the duality relation (\ref{4.3}). Note that
\begin{align*}
    \lim\limits_{r\to1}\Theta(p,q,r)=\frac{\pi}{\sqrt{q-p}}\in \left[\frac{\pi}{k},\frac{\pi}{k-1}\right),
\end{align*}

If in addition $0<p<2$, we have
\begin{align*}
    \lim\limits_{r\to\infty}\Theta(p,q,r)=\frac{q}{q-p}\cdot\frac{\pi}{2}\in \left(\frac{\pi}{2},\frac{3\pi}{4}\right).
\end{align*}
Then $\Theta(p,q,r)$ equals $\frac{\pi}{m}$ for some $r>1$ if 
$$\frac{2(q-p)}{q}<m<\sqrt{q-p}.$$
Thus there exist at least $(k-1)$ embedded solutions.

If in addition $-1<p\le 0$, we have
\begin{align*}
    \lim\limits_{r\to\infty}\Theta(p,q,r)=\frac{\pi}{2}.
\end{align*}
It follows from Theorem \ref{Thm3.2} that $\Theta(p,q,r)$ is monotone decreasing in $r$ as $r$ tends to infinity. Then $\Theta(p,q,r)$ equals $\frac{\pi}{m}$ for some $r>1$ if 
$$2\le m<\sqrt{q-p}.$$
Thus there exist at least $(k-1)$ embedded solutions.

If in addition $-2<p\le -1$, we have
\begin{align*}
    \lim\limits_{r\to\infty}\Theta(p,q,r)=\frac{\pi}{2}.
\end{align*}
Then $\Theta(p,q,r)$ equals $\frac{\pi}{m}$ for some $r>1$ if 
$$2< m<\sqrt{q-p}.$$
Thus there exist at least $(k-2)$ embedded solutions.

If in addition $p>\frac{q}{1-q}$ and $k=3$, we make use of the result in \cite{CCL21} for the planar case.
\begin{thm}[\cite{CCL21}]\label{Thm}
Eq.(\ref{1.2}) admits an even, non-constant, smooth, uniformly convex, positive solution $u$ if $p<0<q,\ q-p>4$ and $\frac{1}{p}+1<\frac{1}{q}$.
\end{thm}

In summary, we have proved the following results.
\begin{lem}\label{Lem7.7}
In the Case (4), we have
\begin{enumerate}
    \item[Case ($4$)] \quad $(k-1)^2<q-p\le k^2,\quad k\ge3$.
    \begin{enumerate}
         \item[Subcase $2^\circ$] If $q\ge 2,\ p<\frac{2q}{2+q}$ and $-2<p\le -1$, then there exist at least $(k-2)$ embedded solutions.
        
        \item[Subcase $2.1^\circ$] If in addition $p>\frac{q}{1-q}$ and $k=3$, then there exist at least two embedded solutions.
        
        \item[Subcase $3^\circ$] If $q\ge 2,\ p<\frac{2q}{2+q}$ and $p>-1$, then there exist at least $(k-1)$ embedded solutions.
        
        \item[Subcase $6^\circ$] If $p\le -2,\ q>\frac{2p}{2-p}$ and $1\le q<2$, then there exist at least $(k-2)$ embedded solutions.
        
        \item[Subcase $6.1^\circ$] If in addition $q<\frac{p}{1+p}$ and $k=3$, then there exist at least two embedded solutions.
        
        \item[Subcase $7^\circ$] If $p\le -2,\ q>\frac{2p}{2-p}$ and $q<1$, then there exist at least $(k-1)$ embedded solutions.

    \end{enumerate}
\end{enumerate}
\end{lem}
Notice that $q<2$ and $p>-2$ does not happen in Case (4) because of $q-p>(k-1)^2\ge4$. Thus Subcase $1^\circ$-Subcase $9^\circ$ complete Case (4) in Theorem \ref{Thm1.2}.

$\ $

Theorem \ref{Thm1.2} follows immediately from the above lemmas. Therefore, we complete the proof of Theorem \ref{Thm1.2}.

\begin{proof}[Proof of Theorem \ref{Thm1.3}]
By the similar observation of Ben Andrews \cite{Ben03}, if there exists a non-constant solution $u(\theta)$ with total curvature $2\pi n$ and $m$ maxima of $u$, then its total curvature of the portion between consecutive critical points must be $\frac{\pi n}{m}$, which equals $\Theta(p,q,r)$ for some $r>1$. Applying Theorem \ref{Thm3.1}, Corollary \ref{Cor4.3} and Theorem \ref{Thm6.4}, the argument proceeds as the proof of Theorem \ref{Thm1.2}. We complete the proof of Theorem \ref{Thm1.3}.
\end{proof}

\appendix
\section{Proof of (\ref{3.10})}\label{appendix}

In the Appendix \ref{appendix} we give the proof of (\ref{3.10}), i.e.
\begin{align}\label{a.1}
    \tilde{I}(p,q,\beta,r,z)
    =(q-p)(1-z^2)\left[1+\frac{z}{2}\left(\beta-\frac{2q-p}{3}\right)(r-1)+o(r-1)\right],
\end{align}
where $\tilde{I}(p,q,\beta,r,z)$ is given by
\begin{equation*}
\begin{split}
    & \tilde{I}(p,q,\beta,r,z)=\frac{4\beta^2}{(r^\beta-1)^2}\left[\left(\frac{r^p-r^q}{r^p-1}v^{\frac{q(\beta-1)}{\beta}}+\frac{r^q-1}{r^p-1}v^{\frac{p+q(\beta-1)}{\beta}}\right)^{\frac{2}{q}}-v^2\right],\quad p\neq0,\\
    & \tilde{I}(0,q,\beta,r,z)=\frac{4\beta^2}{(r^\beta-1)^2}\left[v^{\frac{2(\beta-1)}{\beta}}\left(1+\frac{r^q-1}{\beta\log r}\log v\right)^{\frac{2}{q}}-v^2\right].
\end{split}
\end{equation*}
and 
\begin{align*}
    v=\frac{r^\beta-1}{2}z+\frac{r^\beta+1}{2},\quad -1<z<1.
\end{align*}
Since $\tilde{I}(p,q,\beta,r,z)$ is continuous in $p$, we may assume $p\neq0$ in the proof of (\ref{3.10}).

Let $\delta=r-1$ and $y=\frac{1+z}{2}\in (0,1)$, we have $v=1+(r^\beta-1)y$. When $r$ is close to $1$, a direct computation shows
\begin{align*}
    & \frac{4\beta^2}{(r^\beta-1)^2}=\frac{4}{\delta^2}\left[1-(\beta-1)\delta+o(\delta)\right],\\
    & \frac{r^q-1}{r^p-1}=\frac{q}{p}\left[1+\frac{q-p}{2}\delta+\frac{(p-q)(p-2q+3)}{12}\delta^2+o(\delta^2)\right],
\end{align*}
and
\begin{align*}
    v^{\alpha}
    &=1+\alpha\beta y\delta+\left(\frac{\alpha\beta(\beta-1)}{2}y+\frac{\alpha(\alpha-1)\beta^2}{2}y^2\right)\delta^2\\
    &\quad +\left(\frac{\alpha\beta(\beta-1)(\beta-2)}{6}y+\frac{\alpha(\alpha-1)\beta^2(\beta-1)}{2}y^2+\frac{\alpha(\alpha-1)(\alpha-2)\beta^3}{6}y^3\right)\delta^3+o(\delta^3),
\end{align*}
for any exponent $\alpha$.

First we have 
\begin{align}\label{a.2}
    \frac{r^p-r^q}{r^p-1}v^{\frac{q(\beta-1)}{\beta}}+\frac{r^q-1}{r^p-1}v^{\frac{p+q(\beta-1)}{\beta}}
    =1+q\beta y\delta+E\delta^2+F\delta^3+o(\delta^3),
\end{align}
as $\delta$ tends to $0$, where
\begin{align*}
    E=\frac{qy}{2}\left(q-p+\beta(\beta-1)\right)+\frac{qy^2}{2}\left(p-q+\beta^2(q-1)\right),
\end{align*}
and
\begin{align*}
    F &=qy\left(\frac{\beta(\beta-1)(\beta-2)}{6}+\frac{(q-p)(3\beta-6-p+2q)}{12}\right)\\
    &\quad +qy^2\left(\frac{(\beta-1)^2(q(\beta-1)-\beta)}{2}+\frac{(p+2q(\beta-1)-\beta)(q-p+2\beta-2)}{4}\right)\\
    &\quad +y^3\left(\frac{(q-p)(\beta-1)(q(\beta-1)-\beta)(q(\beta-1)-2\beta)}{6}\right.\\
    &\quad \left.+\frac{q(p+q(\beta-1))(p+q(\beta-1)-\beta)(p+q(\beta-1)-2\beta)}{6p}\right).
\end{align*}

Next, it follows from (\ref{a.2}) that
\begin{align*}
    \uppercase\expandafter{\romannumeral2}
    &:=\left(\frac{r^p-r^q}{r^p-1}v^{\frac{q(\beta-1)}{\beta}}+\frac{r^q-1}{r^p-1}v^{\frac{p+q(\beta-1)}{\beta}}\right)^{\frac{2}{q}}-v^2\\
    &=\left(\frac{2}{q}E-\beta(\beta-1)y+(1-q)\beta^2 y^2\right)\delta^2
    +\left(\frac{2}{q}F+\frac{2(2-q)}{q}\beta y E-\frac{\beta(\beta-1)(\beta-2)}{3}y\right.\\
    &\quad \left.-(\beta-1)\beta^2 y^2+\frac{(2-q)(2-2q)}{3}\beta^3 y^3\right)\delta^3+o(\delta^3)\\
    &=G\delta^2+H\delta^3+o(\delta^3)
\end{align*}
as $\delta$ tends to $0$, where $G=(q-p)y(1-y)$, and
\begin{align*}
    H &=y\frac{(q-p)(3\beta-6-p+2q)}{6}+y^2\frac{(q-p)(p-2q+\beta+2)}{2}\\
    &\quad +y^3\frac{(q-p)(2q-p-3\beta)}{3}.
\end{align*}

Finally, we obtain
\begin{align*}
    \tilde{I} &=\frac{4\beta^2}{(r^\beta-1)^2}\cdot\uppercase\expandafter{\romannumeral2}
    =4\left(1-(\beta-1)\delta+o(\delta)\right)\left(G+H\delta+o(\delta)\right)\\
    &=4\left(G+(H-(\beta-1)G)\delta+o(\delta)\right)\\
    &=4(q-p)y(1-y)+2(q-p)\left(\beta-\frac{2q-p}{3}\right)y(1-y)(2y-1)\delta+o(\delta)\\
    &=(q-p)(1-z^2)\left[1+\frac{z}{2}\left(\beta-\frac{2q-p}{3}\right)\delta+o(\delta)\right],
\end{align*}
as $\delta$ tends to $0$. This completes the proof of (\ref{3.10}).

$\ $

$\ $

\end{document}